\documentclass[12pt]{amsart}
\usepackage{latexsym}
\usepackage{amssymb}
\usepackage{amsmath}
\usepackage{bbm}

\newtheorem{thm}{Theorem}[section]
\newtheorem{lem}[thm]{Lemma}
\newtheorem{cor}[thm]{Corollary}
\newtheorem{prop}[thm]{Proposition}

\newtheorem{prob}[thm]{Problem}

\theoremstyle{remark}
\newtheorem{rem}[thm]{Remark}

\numberwithin{equation}{section}

\newcommand{\beq}{\begin{equation}}
\newcommand{\eeq}{\end{equation}}

\def\real{\hbox{\rm\setbox1=\hbox{I}\copy1\kern-.45\wd1 R}}
\def\prob{\hbox{\rm\setbox1=\hbox{I}\copy1\kern-.45\wd1 P}}
\def\natural{\hbox{\rm\setbox1=\hbox{I}\copy1\kern-.45\wd1 N}}
\def\integers{\hbox{\rm\setbox1=\hbox{-}\copy1\kern-1.2\wd1 Z}}

\def\diam{\mathop{\rm diam}}

\newcommand{\B}{\mbox{$\mathcal B$}}

\allowdisplaybreaks

\begin{document}
\title{Moving Averages}
\author{Terrence Adams and Joseph Rosenblatt}
\date{February 2023}

\begin{abstract} 
We consider the convergence of moving averages in the general setting 
of ergodic theory or stationary ergodic processes.  We characterize 
when there is universal convergence of moving averages based 
on complete convergence to zero of the standard ergodic averages.  
Using a theorem of Hsu-Robbins (1947)
for independent, identically distributed processes, we give a concise 
proof that given a compact standard probability space $(X,\B,\mu)$, 
a continuous function $f:X\to \real$ and 
a certain family of Bernoulli transformations $T$, 
then all moving averages 
$M(v_n, L_n)^T f = 
\frac{1}{L_n} \sum_{i=v_n+1}^{v_n+L_n} f \circ T^i$ with $L_n$ 
strictly increasing converge pointwise 
for almost every $x\in X$:
\[
\lim_{n\to \infty} \frac{1}{L_n} \sum_{i=v_n+1}^{v_n+L_n} f(T^i x) 
= \int_X f d\mu . 
\] 
The process $f\circ T^i$ may be far from independent and identically 
distributed (e.g., $f$ is a 1-1 function).  
This result is extended to show for any bounded measurable function 
$f$, there exists a Bernoulli transformation $T$ such that all 
moving averages $M(v_n, L_n)^T f$ converge pointwise. 

We refresh the reader about the 
cone condition established 
by Bellow, Jones, Rosenblatt (1990) which guarantees 
convergence of certain moving averages for all $f \in L^1(\mu)$ 
and ergodic measure preserving maps $T$.  We show given $f \in L^1(\mu)$ 
and ergodic measure preserving $T$, there exists a moving average 
$M(v_n,L_n)^T f$ with $L_n$ strictly increasing such that $(v_n,L_n)$ 
does not satisfy the cone condition, but pointwise 
convergence holds almost everywhere.  

We show for any non-zero function $f\in L^1(\mu)$, there is a generic 
class of ergodic maps $T$ such that each map has an associated moving 
average $M(v_n, L_n)^T f$ which does not converge pointwise.  
However, we already know from previous work that if $f\in L^2(\mu)$ 
is mean-zero, then there exist solutions $T$ and $g\in L^1(\mu)$ 
to the coboundary equation: $f = g - g\circ T$.  
This implies that $f$ and $T$ produce universal moving averages. 
We show this does not generalize to $L^p(\mu)$ for $p<2$ 
by constructing a family $\mathcal{F}_2$ of functions such that 
$\mathcal{F}_2\cap_{p<2} L^p(\mu)\neq \emptyset$, and 
for $f\in \mathcal{F}_2$ and each ergodic measure preserving $T$, 
there exists a moving average $M(v_n, L_n)^T f$ with $L_n$ 
strictly increasing such that ergodic averages do not converge pointwise.  
Several of the results are generalized to the case of moving averages 
with polynomial growth. 
\end{abstract}


\maketitle

\section {\bf Introduction}\label{intro}

We assume that we have a standard, non-atomic probability space $(X,\B,\mu)$.  We consider measure-preserving maps $T$ of $(X,\B,\mu)$ and functions in $f\in L^p(\mu)$ for $p\geq 1$.  
Given a sequence of integer pairs, $(v_n, L_n)$, define the 
{\em moving average with respect to $T$} of $f$ as, 
\[
M(v_n,L_n)^Tf =  \frac 1{L_n}\sum\limits_{i=0}^{L_n-1} f \circ T^{v_n+i} . 
\]

For a fixed map $T$, we give a characterization of $f\in L^1(\mu)$ such that there is almost everywhere convergence for all moving averages with $L_n$ strictly increasing.  Since we are mostly interested in ergodic averages where $L_n$ is strictly increasing, we will use the term increasing to mean strictly increasing.  The necessary and sufficient condition is that the usual ergodic averages $M(0,n)^Tf = \frac 1n\sum\limits_{k=1}^n f\circ T^k$ have $\sum\limits_{n=1}^\infty \mu \{|M(0,n)^Tf| \ge \delta\} < \infty$ for all $\delta > 0$.  This series convergence is called {\em complete convergence} to $0$, or just complete convergence, of $M(0,n)^Tf$.  It holds for a first category linear subspace of functions in $L^1(\mu)$, but a subspace that is always larger than just the linear subspace of $T$-coboundaries with transfer functions in $L^1(\mu)$.  

But also, if we fix $f$, we show that for a generic class of maps $T$, the function fails to allow a.e. convergence for some moving average with respect to $T$.  So there is the challenging problem of finding $T$ for which the moving averages of $f$ always converge almost everywhere. 
Two main approaches are considered in this paper.  

First, we leverage a result of Hsu-Robbins on complete convergence 
of independent identically distributed processes.  The Hsu-Robbins 
Theorem together with our characterization for universal moving averages 
implies if a function $f\in L^2(\mu)$ admits a transformation $T$ such that 
$X_i=f\circ T^i$ is i.i.d., then all moving averages $M(v_n,L_n)^T$ 
with $L_n$ increasing converge a.e.  
Also, we show that if $f$ is continuous on a compact space $X$, 
then for any Bernoulli transformation $T$, all moving averages 
converge where $L_n$ is increasing.  
In this case, it is not necessary that $X_i=f\circ T^i$ 
is i.i.d.  

Second, we observe that finding a transformation $T$ with universally convergent moving averages is related to the coboundary existence problem.  This is the problem of expressing a measurable function $f = g - g \circ T$ with respect to another measurable function $g$ and a measure-preserving map $T$.  When this occurs, we say that $f$ is a {\em $T$-coboundary with transfer function $g$}.  In Adams and Rosenblatt~\cite{AR3}, we have shown that for $f\in L^p(\mu)$, $1 \le p \le \infty$, there is an ergodic map $T$ and function $g \in L^{p-1}(\mu)$ such that the function 
$f - \int_X f d\mu$ is a $T$-coboundary with transfer function $g$.  In the case where $g\in L^1(\mu)$, this implies that all moving averages $M(v_n,L_n)^T$ converge to $\int_X f d\mu$ pointwise where $L_n\geq n$.  Thus, for $f\in L^2(\mu)$, this coboundary result guarantees an ergodic transformation $T$ with universal moving averages.  Also, this connection extends to $g\in L^p(\mu)$ for $p < 2$.  In the appendix, we show that if $f = g - g\circ T$ where $g \in L^p$ for $p>0$, then given $(v_n, L_n)$ where $L_n \geq n^{\frac{1}{p}}$, the following limit converges pointwise a.e.:
\[
\lim_{n\to \infty} \frac{1}{L_n} \sum_{i=0}^{L_n-1} f(T^{v_n+i}x) = 0 . 
\]

This raises the question of whether given $f\in L^1(\mu)$, there exists 
ergodic $T$ with universal moving averages.  In the final section, 
we construct ``bad'' functions $f \in \bigcup_{p<2} L^p(\mu)$ such 
that given any ergodic $T$, there exists a ``bad'' moving average 
$M(v_n,L_n)^T$ where $L_n \geq n$.  The main results are extended 
to moving averages with polynomial lengths $L_n = n^d$ for 
$d \in \mathbb Z^+$.  

\section {\bf Background on Moving Averages}\label{IntroMovAve}

Consider a sequence of pairs $(v_n,L_n)$ where $(v_n)$ is a sequence in $\mathbb Z$ and $(L_n)$ is an increasing sequence in $\mathbb Z^+$.  The associated moving averages are the continuous, linear operators  on $L^1(\mu)$ given by $M(v_n,L_n)^Tf = \frac 1{L_n}\sum\limits_{k=v_n+1}^{v_n+L_n} f \circ T^k$ for a function $f\in L^1(\mu)$.  Notice that $M(0,n)^T (f) = \frac 1n\sum\limits_{k=1}^n f \circ T^k$ is the classical ergodic average for $T$ and $f$.

Since we are assuming that $L_n$ is increasing, we have $L_n\to \infty$.  If $T$ is ergodic, it is not difficult to see using $T$-coboundaries that these averages converge to $\int_X f\, \, d\mu$ in $L^p$-norm for all $f \in L^p(\mu), 1 \le p < \infty$.  Indeed, suppose $f = g - g\circ T$ is a coboundary with a {\em transfer function} $g \in L^\infty(\mu)$.  Then $\|M(v_n,L_n)^Tf\|_\infty \le 2\|g\|_\infty/L_n$.  So for this class of $T$-coboundaries, it is clear that $M(v_n,L_n)^Tf$ converges to zero in $L^\infty$-norm, and hence also in $L^p$-norm for all $p, 1 \le p \le \infty$.  For $1 \le p < \infty$, this class of $T$-coboundaries is $L^p$-norm dense in the orthogonal complement of the $T$-invariant functions in $L^p(\mu)$, which are the constant functions if $T$ is ergodic. Hence, the moving averages always converge in $L^p$-norm, and if $T$ is ergodic they always converge in $L^p$-norm to $\int_X f \, \, d\mu$ for all $f\in L_p(\mu), 1 \le p < \infty$.

However, despite norm convergence being easily resolved, pointwise almost everywhere (a.e.) behavior is much more complicated.  The main theorem in Bellow, Jones, and Rosenblatt~\cite{BJR} characterizes when these averages converge a.e. for all $f\in L^1(\mu)$.    The criterion is a geometric counting condition called the {\em Cone Condition}.  One takes the vertical cones in $\mathbb R^2$ with  vertex at $(v_n,L_n)$ and a vertex angle of $90$ degrees.  Then for a given $\lambda > 0$,  let $C(\lambda)$ be the one-dimensional Lebesgue measure of the set of $(x,\lambda)$ with $x\in \mathbb R$ which are in at least one of these cones.  The Cone Condition states that for some constant $K$, $C(\lambda) \le K \lambda$ for all $\lambda > 0$.  The  basic theorem is that there is  a.e. convergence of the moving averages $M(v_n,L_n)^Tf$ for all $f\in L^1(\mu)$ and all $T$ if and only if the Cone Condition holds.  Also, the Cone Condition holds if and only if there is a.e. convergence of the  moving averages on the smallest Lebesgue space $L^\infty(\mu)$ for some ergodic map $T$, and when the Cone Condition fails, then there is {\em strong sweeping out} for the moving averages.  See \cite{BJR} for details.

For example, the Cone Condition fails for $(n^2, n)$ as the pairs $(v_n,L_n)$.   It also fails for a rare subsequence of these, e.g.  using just of the pairs $(4^n,2^n)$ with $n \ge 1$.  However, once one takes a hyper-lacunary subsequence, then Cone Condition holds.  Indeed, taking $(v_n,L_n) = (2^{2^{n+1}},2^{2^n})$ for $n \ge 1$, the averages $(M(v_n,L_n)^Tf)$ converge a.e. for all of $L^1(\mu)$.

Given this, the interesting case is when the Cone Condition fails.  Then for most functions we do not have a.e. convergence.  Indeed, it can be shown that there is a dense, $G_\delta$ set of functions $f\in L^1(\mu)$ such that $\limsup\limits_{n\to \infty} |M(v_n,L_n)^Tf| = \infty$ a.e.  These are {\em bad functions} i.e. functions for which a.e. convergence fails.   But what can be said about
the {\em good functions}, the subset of the complement of the bad functions for which there is a.e. convergence to $\int_X f\, \, d\mu$?  This is empirically an important question if one thinks of this averaging process in applications.  Then any choice of a function (in some sense the template for the data) is essentially only a rough approximation of the general case of bad functions. In fact, it is perhaps actually a good function, for which a.e. convergence occurs!  Indeed, as above if $f$ is a $T$-coboundary $f = g - g\circ T$ with $g \in L^\infty(\mu)$, then we certainly do have a.e. convergence, and this class is dense in the mean-zero functions in $L^1(\mu)$.

In this article, we consider a variety of questions related to characterizing the class of good functions when the Cone Condition fails.  For example, fix the ergodic map $T$.  In the case of a moving average where the Cone Condition fails, what is a (constructive) characterization of the good functions: the functions $f\in L^1(\mu)$ for which $(M(v_n,L_n)^Tf)$ converges a.e.?  Is it always more than just the coboundary functions?  How does it depend on the nature of the mapping $T$?  Moreover, what happens in the polar case, if instead of fixing $T$, we fix $f$ and let $T$ vary?

\medskip

\section {\bf Basics of Moving Averages}\label{basics}

Consideration of moving averages in ergodic theory was given its legitimate start in the article by Akcoglu and del Junco~\cite{AdJ}, which corrected a claim in Belley~\cite{Belley} that all moving averages converge a.e.
Akcoglu and del Junco showed that for any ergodic $T$, there is some measurable set $E$ such that the averages $M(n^2,n)^T 1_E$ fail to converge a.e.

More generally, here are some basic properties of moving averages that are worth keeping in mind.  These properties will be important for the constructions used in this article.  These results all follow from the Cone Condition.  Some are specifically noted in ~\cite{BJR}.

\begin{prop}  Given a moving average $M(v_n,L_n)^T$, there is a subsequence $(v_{n_m},L_{n_m})$ which satisfies the Cone Condition.  Hence, the moving averages $(M(v_{n_m},L_{n_m})^T (f))$ converge a.e. to $\int_X f\, \, d\mu$ for every ergodic map $T$ and every $f\in L^1(\mu)$.
\end{prop}

\begin{rem}  The contrast here with other norm convergent stochastic processes is clear.  There would always be a subsequence, depending on $f$, for which there would be a.e. convergence.  But here there is a subsequence which makes every $f\in L^1(\mu)$ good for every map $T$.  As a contrast, take a sequence $(x_n)$ of non-zero real numbers with
$\lim\limits_{n\to \infty} x_n = 0$.  Consider the operators $T_n(f)(x) = f(x_n+x)$ for all $f\in L^1(\mathbb R)$.  These converge in $L^1$-norm.  But every subsequence is {\em strongly sweeping out}.  Here strongly sweeping out for a stochastic process $(T_n)$ means that there is a measurable set $A$ such that $\limsup\limits_{n\to \infty} T_n(1_A) = 1$ a.e. and $\liminf\limits_{n\to \infty} T_n(1_A) = 0$ a.e. The same would be true if $T_n$ is the Ces\`aro averages of these translations.  For history and information about these results, see Akcoglu et al~\cite{AdelJLee}, Bellow~\cite{Bellow}, and Bourgain~\cite{Bourgain}.
\qed \end{rem}

Another important principle is this one.

\begin{prop} \label{spreadbad} Given any sequence $(L_n)$ in $\mathbb Z^+$, there exists some $(v_n)$ such that the moving averages $M(v_n,L_n)^T$ fails the Cone Condition.
\end{prop}

It is a central fact that for moving averages the $L^p$-space is in some sense not important.

\begin{prop}\label{allornone} Given an ergodic map $T$ and a moving average given by $(v_n,L_n )$, if there is some $p_0, 1 \le p_0 \le \infty$ such that all functions in $L^{p_0}(\mu)$ are good, then the same is true for all other $p, 1 \le p \le \infty$ and all other ergodic maps.
\end{prop}

\begin{rem} For general reasons, it is not a surprise that what happens for one ergodic map is mirrored by what happens for the full class of ergodic maps.  But as for the $L^p$-space, this result is quite different from the ones that appear in Bellow~\cite{Bellow2} and Reinhold~\cite{Reinhold} where the averages are for powers of a map, but the good $L^p$-spaces depend in an essential fashion on which powers are being used.
\qed \end{rem}

In addition, when the Cone Condition fails, the following extreme form of the failure of a.e. convergence occurs: the generic set is strongly sweeping out.

\begin{prop}\label{swo} If the Cone Condition fails, then for all ergodic maps, the generic measurable set $A \in \beta$ has $\limsup\limits_{n\to \infty} M(v_n,L_n)^T1_A = 1$ a.e. and
$\liminf\limits_{n\to \infty} M(v_n,L_n)^T 1_A = 0$ a.e.
\end{prop}

\section {\bf When the map is fixed}\label{basicsmap}

In this section we consider results where we have fixed the ergodic map $T$, and look to understand the conditions on the function that makes it good.  If the Cone Condition occurs, then it does not matter what $T$ is or what $f$ is, since the associated moving averages will always converge a.e.  It is when we have the unhappy case that the Cone Condition fails that interesting results emerge. 

First, Proposition~\ref{spreadbad} gives immediately the following showing that every $f$ is good for some bad moving average.

\begin{prop}\label{allfgood}  Suppose $f\in L^1(\mu)$ is mean-zero and $T$ is ergodic.  Fix $(L_m)$.  Then there exists $(L_{m_n})$ such that for all $(v_n)$, the moving averages $M(v_n,L_{m_n})^T f$ converge a.e. to $0$.  With appropriate choice of $(v_n)$ this gives $(v_n,L_{m_n})$ which fails the Cone Condition but $f$ is good for that moving average.
\end{prop}
\begin{proof}  Since $M(0,L_m)f$ converges to $0$ in $L^1$-norm, there exists $(L_{m_n})$ such that
\[\sum\limits_{n=1}^\infty \|M(0,L_{m_n})^T f\|_1 < \infty.\]
But then
for any $(v_n)$, we have
\[\|M(v_n,L_{m_n})^T f\|_1 = \|M(0,L_{m_n})^T f\|_1.\]
Hence, $\sum\limits_{n=1}^\infty \|M(v_n,L_{m_n})^T f\|_1 < \infty$ too, and so
$\sum\limits_{n=1}^\infty |M(v_n,L_{m_n})^T f(x)|< \infty$ for a.e. $x$.  So
$M(v_n,L_{m_n})^T f(x) \to 0$ as $n\to \infty$ for a.e. $x$.  We use Proposition~\ref{spreadbad} to complete the result.
\end{proof}

\begin{rem}  The choice of $(L_{n_m})$ here depends on $T$.  Indeed, for a given $f$ and any fixed sequence $(L_n)$, including of course the subsequence chosen here, there is a dense $G_\delta$ set of ergodic maps $\sigma$ such that $\sum\limits_{n=1}^\infty \|M(0,L_{m_n})^\sigma f\|_1 = \infty$.  But moreover, if we have fixed $(L_n)$ (or say $(L_{m_n}))$, there is actually a dense $G_\delta$ set of ergodic maps $\sigma$ such that for each one there is some $(v_n)$ such that the function $f$ is bad for the moving average $M(v_n,L_n)^T\sigma f$.  See Section~\ref{basicsdata}.
\qed \end{rem}

In addition, for a given ergodic map $T$, we have already observed that the $T$-coboundaries provide a dense class of good functions for any choice of $(v_n,L_n)$.    We could use either transfer functions in $L^\infty(\mu)$ or in the larger class $L^1(\mu)$.  The difference is that with bounded transfer functions the moving averages always converge in $L^\infty$-norm to $0$.  However, if the transfer function is merely integrable, we need a different argument.  We need to have $\frac 1{L_n}(f \circ T^{v_n+1} - f\circ T^{v_n+L_n+1}) \to 0$ a.e. and the easiest way to guarantee that without more information about $f$, other than that it is integrable, is to use the fact that for any $(w_n)$,
\[\sum\limits_{n=1}^\infty \mu \{|f\circ T^{w_n}| \ge L_n\} = \sum\limits_{n=1}^\infty \mu \{|f| \ge L_n\} \le \sum\limits_{n=1}^\infty \mu \{|f| \ge n\} < \infty.\]

However, there are always more functions than the $T$-coboundaries that are good functions.  See the article by Derriennic and Lin~\cite{DL}.   In particular, Corollary 2.15 in ~\cite{DL} gives for any $\alpha, 0 < \alpha < 1$, there is $f \in L^p(\mu), 1 \le p < \infty$ such that $\lim\limits_{n\to \infty} n^\alpha\|M(0,n)^Tf\|_p = 0$.

\begin{prop}\label{normseriesconv}  For any ergodic map $T$ and $p, 1 < p < \infty$, there is $f \in L^p(\mu)$ such that
 $\sum\limits_{n=1}^\infty \|M(0,n)^Tf\|_p^p < \infty$. Moreover, there are such functions that are not $T$-coboundaries.  These functions are good for all moving averages with respect to $T$.
\end{prop}

\begin{proof}  We just take $\alpha $ such that $\sum\limits_{n=1}^\infty (1/n^{p\alpha}) < \infty$.  Indeed, for $p > 1$, if we have $1/p < \alpha < 1$, this condition is satisfied.  Now using \cite{DL}, take $f \in L^p(\mu)$ such that $\lim\limits_{n\to \infty} n^\alpha\|M(0,n)^Tf\|_p = 0$, and so $\sum\limits_{n=1}^\infty \|M(0,n)^Tf\|_p^p < \infty$.  But then as in the proof of Proposition~\ref{allfgood}, this means that for any increasing $(L_n)$ and any $(v_n)$, such functions $f$ are good for the moving average $M(v_n,L_n)^Tf$.  In addition, Proposition 2.2 in ~\cite{DL} shows that there are functions produced with these conditions that are not $T$-coboundaries.
\end{proof}

\begin{rem}
The only drawback to the constructions in ~\cite{DL} is that in some ways the examples are related to coboundaries because they are produced using the clever idea of studying
{\em fractional coboundaries}.  But there are probably other constructions that would give even larger classes of functions that are good for the given $T$ and any moving average.
\qed \end{rem}

It turns out that we can characterize the good mean-zero functions here using a series condition.  They are just the functions for which $M(0,n)^Tf$ converges completely to zero.
\begin{prop} \label{charactergood} Suppose $T$ is ergodic.  Then a mean-zero $f\in L^1(\mu)$ is good for all moving averages $(M(v_n,L_n)^Tf)$ with increasing $L_n$ if and only if for all $\delta > 0$, we have
\[\sum\limits_{n=1}^\infty \mu \{|M(0,n)^Tf| \ge \delta\} < \infty.\]
\end{prop}

To prove this result we need a divergent series lemma.  This first appeared in Sawyer~\cite{Sawyer}; see Lemma 2.  We present the proof so that it is clear how it works with an {\em ergodic family} being the powers of one ergodic mapping.

\begin{lem}\label{Div} Suppose $T$ is ergodic and $\sum\limits_{n=1}^\infty \mu(E_n) = \infty$.  Then there exists an increasing sequence $(m_n)$ in $\mathbb Z^+$ such $\sum\limits_{n=1}^\infty 1_{T^{m_n}E_n} = \infty$ a.e.  That is, for a.e. $x$,  we have $x \in T^{m_n}E_n$ infinitely often.
\end{lem}

\begin{proof}
For convenience, assume $\mu(E_n) > 0$ for all $n$. 
Apply the ergodicity of $T$ to construct an increasing sequence 
$N_j$ and corresponding $m_{j,i}$ such that 
\[
\mu \Big( \bigcup_{i=1}^{N_{j+1}-N_j} T^{m_{j,i}}E_{N_j+i} \Big) 
> 1 - \frac{1}{j} . 
\]
Suppose $N_j, m_{j,i}$ are defined for $1\leq i < k$ for some $k$.  
Using the ergodicity of $T$, we can choose $m_{j,k}$ such that 
\[
\mu \Big( T^{m_{j,k}} E_{N_j+k} \cap 
\big( \bigcup_{i=1}^{k-1} T^{m_{j,i}} E_{N_j+i} \big)^c \Big) \geq 
\frac{1}{2} \mu \big( E_{N_j+k} \big) 
\mu \Big( \big( \bigcup_{i=1}^{k-1} T^{m_{j,i}} E_{N_j+i} \big)^c \Big) . 
\]
Thus, 
\[
\mu \Big( \big( \bigcup_{i=1}^{k} T^{m_{j,i}} E_{N_j+i} \big)^c \Big) \leq 
\prod_{i=1}^{k} \Big( 1 - \frac{1}{2} \mu \big( E_{N_j+i} \big) \Big) . 
\]
Convergence to zero of this infinite product is implied 
by the divergence of the corresponding infinite sum: 
$\sum_{i=1}^{\infty} \frac{1}{2}\mu(E_{N_j+i}) = \infty$.  
See \cite[p.~220]{Knopp54} for a reference.  
Thus, 
\[
\prod_{i=1}^{\infty} \Big( 1 - \frac{1}{2} \mu \big( E_{N_j+i} \big) \Big) 
= 0 . 
\] 
Hence, there exists $N_{j+1}=N_{j}+k$ such that 
\[
\mu \Big( \big( \bigcup_{i=1}^{N_{j+1}-N_j} T^{m_{j,i}} E_{N_j+i} \big)^c \Big) 
< \frac{1}{j} . 
\]
Define 
\[
D_j = \bigcup_{i=1}^{N_{j+1}-N_j} T^{m_{j,i}} E_{N_j+i} . 
\]
Then the limsup of the sets $D_j$, 
$D = \cap_{k=1}^{\infty} \bigcup_{j=k}^{\infty} D_j$ is a set of full measure 
and for $x\in D$, there exist infinitely many $j$ such that 
$x\in D_j$.  This implies there exist infinitely many $j$ and $i$ such 
that $x \in T^{m_{j,i}}E_{N_j+i}$.  
It is clear from the method of the construction that $(m_n)$ can be chosen to be increasing.
\end{proof}

\begin{proof} [Proof of Proposition~\ref{charactergood}]  First, suppose the series condition holds.  Then for any $(v_n)$,
\[\sum\limits_{n=1}^\infty \mu \{|M(v_n,L_n)^Tf| \ge \delta\} \le \sum\limits_{n=1}^\infty \mu \{|M(0,n)^Tf| \ge \delta\} < \infty.\]
Therefore, a.e. $x$ is eventually not in $\{|M(v_n,L_n)^Tf| \ge \delta\}$ and so $ |M(v_n,L_n)^Tf(x)| < \delta$.  Since $\delta > 0$ can be
arbitrarily small, we have $M(v_n,L_n)^Tf(x)$ converges to $0$ a.e.  Hence, for this fixed $T$, $f$ is good for all moving averages.

On the other hand, suppose for some $\delta > 0$, we have
\[\sum\limits_{n=1}^\infty \mu \{|M(0,n)^Tf| \ge \delta\} = \infty.\]
We show that $f$ is not good for some moving average with respect to $T$.
First, by Lemma~\ref{Div} applied to $T^{-1}$, there exists $v_n$ so that for a.e. $x$, we have $T^{v_n}x \in \{|M(0,n)^Tf| \ge \delta\}$ infinitely often.  Hence, using this $(v_n)$, for a.e. $x$ we have  $|M(v_n,n)^T f(x)| \ge \delta$ infinitely often.  However, since $f$ is mean-zero, $(M(v_n,n)^T f)$ converges in $L^1$-norm to zero and so there is a subsequence of these moving averages which converges to $0$ a.e.  So $f$ is not good, in fact it is bad, for this moving average.
\end{proof}

The proof is showing this.

\begin{cor}  Suppose $f\in L^1(\mu)$ is not good for some moving average with respect to an ergodic $T$.  There there is a moving average with respect to $T$ such that the function is actually bad.
\end{cor}

\begin{rem}  Proposition~\ref{charactergood} is showing that if $f$ is not good for some moving averages with respect to $T$, then for some $(v_n)$ and $\delta > 0$, $\limsup\limits_{n\to \infty} |M(v_n,n)^T f| \ge \delta$ a.e.  But we also know that $\liminf\limits_{n\to \infty} |M(v_n,n)^T f| = 0$ a.e.  and this is what shows the failure of convergence a.e.  To deny the function being good, it would have been enough to show just that $\limsup\limits_{n\to \infty} |M(v_n,n)^T f| \ge \delta$ on some set of positive measure.  But because $\lim\limits_{n\to \infty} M(v_n,n)^T (f - f \circ T) = 0$ a.e., actually $\limsup\limits_{n\to \infty} M(v_n,n)^T f$ is $T$-invariant. So it is a constant since $T$ is assumed to be ergodic.  Hence actually $\limsup\limits_{n\to \infty} |M(v_n,n)^T f|$ is a constant a.e.  We know that it can be $\infty$ a.e. Indeed when the Cone Condition fails, this happens for the generic function $f$.  We suspect that for any $C, 0 < C \le \infty$, there is a mean-zero function that is not good such that for some $(v_n)$,  $C = \limsup\limits_{n\to \infty} |M(v_n,n)^T f|$.  Indeed, for $0 < C \le 1$, it is probably the case that there is a characteristic function $f = 1_E$ which is not good such that for some $(v_n)$, $C = \limsup\limits_{n\to \infty} |M(v_n,n)^T (1_E- \mu(E))|$.
\qed \end{rem}

There is also a variation via subsequences for classifying the functions that are good for all moving averages, one that replaces $M(v_n,n)^Tf$ by $M(v_n,L_n)^Tf$ for some fixed $(L_n)$.    Of course, Proposition~\ref{allfgood} also shows that such liberality of considering all $(L_n)$ will capture all functions as good, depending on the choice of $(L_n)$.  Sometimes we can even determine explicitly what $(L_n)$ to use as in the case of $T$-fractional coboundaries.
In any case, here is the result.

\begin{prop}\label{Divgen}  Suppose $(L_n)$ is a fixed increasing sequence in $\mathbb Z^+$.  For an ergodic map $T$, a mean-zero function $f\in L^1(\mu)$ is good for all moving averages $(M(v_n,L_n)^T)$ if and only if for all $\delta > 0$, we have
\[\sum\limits_{n=1}^\infty \mu \{|M(0,L_n)^T f| \ge \delta \} < \infty.\]
Also, if $f$ is not good for some such moving average, then there is another moving average $M(v_n,L_n)^Tf$ which a.e. fails to converge, so $f$ is bad for some moving averaging with the lengths of the averages in $(L_n)$.
\end{prop}

\section {\bf The IID case of moving averages}\label{IID}

The series characterization of good functions is a familiar one in probability theory.    To keep our notation consistent, we take an IID sequence $(X_k:k \ge 0)$ and denote by $M(v_n,L_n)^TX_0$ the moving average $\frac 1{L_n}\sum\limits_{k=v_n+1}^{v_n+L_n} X_k$.  In this section, we present the famous result 
of Hsu-Robbins ~\cite{HR}, as well as the converse proved by Erd\H{o}s~\cite{Erdos}.  We also give extensions to this result as well as some limitations. 

\subsection{Hsu-Robbins-Erd\H{o}s result}
\begin{prop} \label{HRE} Suppose $(X_k:k \ge 0)$ is an IID sequence. Then
\[\sum\limits_{n=1}^\infty \mu \{|M(0,n)X_0| \ge \delta \} < \infty\]
for all $\delta > 0$ if and only if $X_0$ is mean-zero and $X_0 \in L^2$.
\end{prop}

\noindent Indeed, with $X_0$ being mean-zero, one has $X_0 \in L^2$ just if
\[\sum\limits_{n=1}^\infty \mu \{|M(0,n)X_0| > 1\} < \infty.\]

We can use Proposition~\ref{HRE} to prove the following corollary.

\begin{cor}\label{GoodIID} Any IID sequence $(X_k)$ with $X_0 \in L^2$ has the property that for all $(v_n)$ the moving averages $M(v_n,n)X_0$ converge a.e. to the mean of $X_0$.
\end{cor}

\subsection{Extensions and limitations to Hsu-Robbins-Erd\H{o}s} 
We begin by proving the following restriction on Corollary~\ref{GoodIID}. 

\begin{prop} For any IID sequence $(X_k)$ with $X_0 \in L^1\backslash L^2$ and $X_0$ mean-zero, if $v_n = n(n+1)/2$, then a.e. the moving averages $M(v_n,n)X_0$ do not converge.
\end{prop}
\begin{proof} Erd\H{o}s~\cite{Erdos} shows that
$\sum\limits_{n=1}^\infty \mu \{|M(0,n)X_0| > 1\} = \infty$.  It is easy to see then that the moving averages $M(v_n,n)X_0=  \frac 1n\sum\limits_{k=v_n+1}^{v_n+n} X_k$ are independent because $v_n =n(n+1)/2 = \sum\limits_{k=1}^n k$.  But also
\[\sum\limits_{n=1}^\infty \mu \{|M(v_n,n)X_0| > 1\} = \sum\limits_{n=1}^\infty \mu \{|M(0,n)X_0| > 1\} = \infty.\]
So the Borel-Cantelli Lemma shows that $\sum\limits_{n=1}^\infty 1_{\{|M(v_n,n)X_0| > 1\}} = \infty$ a.e.   However, in $L^1$-norm, $M(v_n,n)X_0 \to 0$ as $n\to \infty$, and so some subsequence converges a.e. to $0$.  This proves that a.e. the moving averages $M(v_n,n)X_0$ fail to converge.
\end{proof}

It is interesting that the problem with these moving averages is really the amount of data being collected. Here is a simple observation.

\begin{prop} \label{Lp} For any mean-zero IID sequence $(X_k)$ with $X_0 \in L^p, 1 < p < 2$, whenever  $D(p -1) > 1$, then for all $(v_n)$, the moving averages
have $\lim\limits_{n\to \infty} M(v_n,\lfloor n^D\rfloor)X_0 = 0$ a.e.
\end{prop}
\begin{proof}  We use the $L^p$-norm estimate $\|\sum\limits_{k=1}^n X_k\|_p \le 2n^{1/p}\|X_0\|_p$ in Rosenthal~\cite{Rosenthal}, Lemma 2 b).  This gives
$\sum\limits_{n=1}^\infty  \|M(v_n,\lfloor n^D\rfloor)X_0\|_p^p < \infty$ since $D(p-1) > 1$.  This is all that is needed for every moving average with $\lfloor n^D\rfloor$ terms to converge a.e.
\end{proof}

The results of Asmussen and Kurtz~\cite{AK}, see also Gut~\cite{Gut}, address the same issue with more precision, in particular for $L^p, 1 \le p < \infty$.  For example, we have this better result than Proposition~\ref{Lp}.

\begin{prop} \label{Lp2} For any mean-zero IID sequence $(X_k)$ with $X_0 \in L^p, 1 < p < 2$, whenever $D(p -1) = 1$, then for all $\delta > 0$,
\[\sum\limits_{n=1}^\infty \mu \{|M(0,\lfloor n^D\rfloor)X_0| \ge \delta \} < \infty.\]
Hence, for all $(v_n)$, the moving averages
have $\lim\limits_{n\to \infty} M(v_n,\lfloor n^D\rfloor)X_0 = 0$ a.e.
\end{prop}

In addition, we have this result for $L^1$.

\begin{prop}  For any mean-zero IID sequence $(X_k)$ with $X_0 \in L^1$, and any $\delta > 0$,
\[\sum\limits_{n=1}^\infty \mu \{|M(0,2^n)X_0| \ge \delta \} < \infty.\]
Hence, for any mean-zero IID sequence $(X_k)$ with $X_0 \in L^1$, one knows that for all $(v_n)$, the moving averages
have $\lim\limits_{n\to \infty} M(v_n,2^n)X_0 = 0$ a.e.
\end{prop}

\begin{rem} Part of what is shown in Gut~\cite{Gut} is that the rates used for the subsequence results are characteristic of the $L^p$-class.  Also, by the Borel-Cantelli Lemma, the a.e. convergence of all the moving averages will imply the complete convergence for $(M(0,L_n)X_0)$ for appropriate $L_n$, and hence force $X_0$ to be in the corresponding $L^p$-class.

Moreover, an intermediate space like $L\log L$ has its own rate for all moving averages to converge: one should use $L_n = \lfloor \exp (\sqrt n)\rfloor$.  This is exactly what is needed for having every mean-zero function $X_0 \in L\log L$ to have $\sum\limits_{k=1}^\infty \mu \{|M(0,L_n)X_0| \ge \delta \} < \infty$ for all $\delta > 0$.  Hence, again, this is the rate that is needed on $(L_n)$ to have all moving averages converging a.e.
\qed \end{rem}

\subsection {\bf Random selection approach} \label{IIDconstruction}

Here is another approach to answering the conjecture that works for a different class of functions in $L^2(\mu)$.  We use the Hsu and Robbins Theorem~\cite{HR}.
Given a mean-zero $f \in L^2(\mu)$, assume there is an isomorphism $\rho$ from a standard probability space $Y$ to $X$ and there is a map $\alpha$ of $Y$ such that $(f\circ \rho \circ \alpha^k: k \ge 1)$ are IID.  Then if we take the map $T = \rho\circ \alpha \circ \rho^{-1}$, we would have $(f\circ T^k: k \ge 1)$ IID, and so  $f$ will be good for all moving averages using $T$ because of the Hsu and Robbins Theorem.

One could hope that this applies to every mean-zero function $f \in L^2(\mu)$ by considering a function $h$ with the same distribution function as $f$ but for which there is a map $\alpha$ such that the sequence $(h\circ \alpha^k: k \ge 1)$ is IID.  But this is not generally a viable approach because functions with the same distribution are not necessarily invertible, measure-preserving rearrangements of each other.

However, for measurable sets this is true.  That is, if $A_1,A_2$ have the same measure in a standard probability space $Y$, then there is a measure-preserving invertible map $\nu$ of $Y$ such that $\nu (A_1) = A_2$ up to null sets.  We can use this to show that for any measurable $A$ in $X$, the function $f = 1_A - \mu(A)$ is good for all moving averages with respect to some map $T$ on $X$.  Indeed, take the product space $Y = \prod\limits_{-\infty}^\infty  X$ with the product measure given by $m$ in each coordinate.  Let $\alpha$ be the usual coordinate shift mapping on $Y$.  Take the first coordinate projection $\pi:Y \to X$ and let $h = f\circ \pi$.  Take any isomorphism $\omega: X \to Y$.   The sets $\omega (A)$ and $\pi^{-1} A$ in $Y$ have the same measure.  So there is an isomorphism $\nu$ of $Y$ such that $\nu \circ \omega (A) = \pi^{-1} A$ up to null sets.  But then take the isomorphism $\rho$ and consider the mapping $T$ of $X$ given by $T = \rho^{-1}  \circ \alpha\circ \rho $.  Because the sets $(\alpha^k (\pi^{-1} (A)): k \in \mathbb Z)$ are jointly independent, the functions $(h\circ \alpha^k: k \ge 1)$ are IID.  Hence, so are the functions $(f \circ T^k: k \ge 1)$.

Now we apply the Hsu and Robbins Theorem to show that $f$ is good for all moving averages with respect to $T$.
It is possible to extend this construction to mean-zero, simple functions $f$, not just functions of the form $1_A - \mu(A)$.

One can ask when this can be extended to other functions:
\bigskip

\noindent {\bf Question} Given $f \in L^p(\mu), 1 \le p \le \infty$, when is there an (ergodic) map $T$ of $X$ such that $(f\circ T^k: k \in \mathbb Z)$ is IID?
\bigskip

This is not always possible. For example, if the function $f$ has $f^{-1}(\beta) = \beta$, then there is no such map $\sigma$.  For example, take $X = [0,1]$ in Lebesgue measure and let $f(x) = x$ for all $x \in [0,1]$.  B. Weiss has observed that what one needs is something like the $\sigma$-algebra $f^{-1}(\beta_{\mathbb R})$, with $\beta_{\mathbb R}$ being the Borel sets in $\mathbb R$, admits an {\em independent complement} $\mathcal A$ in $\beta$ such that $(X,\mathcal A,m)$ is isomorphic to a standard, non-atomic space too.  Here we are taking $\mathcal A$ to be a sub-$\sigma$-algebra in $\beta$ such that $\mathcal A$ and $f^{-1}(\beta_{\mathbb R})$ are independent.

\subsection{Continuous functions} 
We can apply the Hsu-Robbins theorem together with Proposition 
\ref{charactergood} to prove there are Bernoulli maps which 
produce univerally convergent moving averages in cases 
where the process is not i.i.d.  
Suppose $(X,d)$ is a compact metric space with metric $d$.  
Given a partition $P$ of $X$ and element $p\in P$, define 
\[
\diam(p)=\sup{ \{ d(x,y) : x,y \in p \} } 
\]
and $\diam(P) = \sup_{p\in P} \diam(p)$.  
Given a finite entropy Bernoulli transformation $T$, 
there exists a finite generating partition $P$.  
In particular, for each $n \in \mathbb Z^+$, 
each partition $T^{-n}P$ and $T^n P$ is independent of 
\[
P_n = \bigvee_{i=1-n}^{n-1} T^{i}P 
\]
and the sequence $P_n$ generates the sigma 
algebra $\B$.  
Next we prove given a continuous function $f$ on a compact space, 
there is a large class of Bernoulli transformations such 
that moving averages converge pointwise almost everywhere. 
\begin{thm}
\label{C-of-B}
Suppose $(X, \B, \mu)$ is a standard probability space 
and $(X,d)$ is a compact metric space.  
Let $T$ be a Bernoulli transformation with finite generating 
partition $P$ such that 
\[
\lim_{n\to \infty} \diam(P_n) = 0 . 
\]
If $f$ is a continuous function on $(X,d)$, then 
all moving averages $M(v_n,L_n)^T f$ converge pointwise a.e. 
to $\int_X f d\mu$ for 
$v_n\in \mathbb Z$ and $L_n \in \mathbb Z^+$ with $L_n$ is increasing. 
\end{thm}
\begin{proof}
Although, the process $X_i = f\circ T^i$ may not be i.i.d., 
we can approximate $X_i$ by a finite collection of i.i.d. processes.  
Since $f$ is continuous on a compact space $X$, then 
$f$ is uniformly continuous.  
Thus, given the finite generating partition $P$ such that 
$\diam(P_n) \to 0$ and $\delta > 0$, there 
exists $n \in \mathbb Z^+$ such that for $x\in X$, 
\[
|\mathcal{E}\big( f | P_n )(x) - f(x)| < \frac{\delta}{2} , 
\]
where $\mathcal{E}(f|P_n)(x) = \int_{p}f d\mu$ for $x\in p$, 
$p\in P_n$ (conditional expectation of $f$ given $P_n$).  
Also, since $f$ is bounded, we can choose $N \in \mathbb Z^+$ such that 
\[
\frac{n}{N} ||f||_{\infty} < \frac{\delta}{4} . 
\]
Since $T^{2ni}P_n$ are independent for $i \in \mathbb Z^+$, then 
the process $Y_{0,i}(x) = \mathcal{E}(f|P_n)(T^{2ni}x)$ is 
an i.i.d. process.  Moreover, for each $k$, $0\leq k \leq 2n-1$, 
the following processes are each i.i.d., 
\[
Y_{k,i}(x) = \mathcal{E}(f|P_n)(T^{2ni+k}x) . 
\]
Hence, we get 
\begin{align}
\sum_{m=N}^{\infty} \mu \big( \{ x: | \frac{1}{m} \sum_{i=0}^{m-1} 
f(T^i x) | \geq \delta \}\big) &\leq 
\sum_{m=N}^{\infty} \mu \big( \{ x: | \frac{1}{m} \sum_{i=0}^{m-1} 
\mathcal{E}(f|P_n)(T^i x) | \geq \frac{\delta}{2} \}\big) \label{sum-1} \\ 
&\leq  \sum_{m=N}^{\infty} \mu \big( \{ x: | \frac{1}{m} 
\sum_{i=0}^{\lceil \frac{m}{2n}\rceil-1} \sum_{k=0}^{2n-1} 
\mathcal{E}(f|P_n)(T^{2ni+k} x) | \geq \frac{\delta}{4} \}\big) \\ 
&\leq  \sum_{k=0}^{2n-1} \sum_{m=N}^{\infty} \mu \big( \{ x: | \frac{1}{m} 
\sum_{i=0}^{\lceil \frac{m}{2n}\rceil-1} 
\mathcal{E}(f|P_n)(T^{2ni+k} x) | \geq \frac{\delta}{8n} \}\big) \\ 
&\leq  2n \sum_{k=0}^{2n-1} \sum_{m=N}^{\infty} \mu \big( \{ x: | \frac{1}{m} 
\sum_{i=0}^{m-1} 
\mathcal{E}(f|P_n)(T^{2ni+k} x) | \geq \frac{\delta}{8n} \}\big) 
\end{align}
We know from the Hsu-Robbins-Erd\H{o}s theorem that for each 
$k$, $0\leq k < 2n$, 
\[
\sum_{m=1}^{\infty} \mu \big( \{ x: | \frac{1}{m} 
\sum_{i=0}^{m-1} 
\mathcal{E}(f|P_n)(T^{2ni+k} x) | \geq \frac{\delta}{8n} \}\big) 
< \infty . 
\]
Therefore, the sum (\ref{sum-1}) is finite and the theorem follows 
from Proposition~\ref{charactergood}. 
\end{proof}

Note, it is not possible to have all Bernoulli transformations 
produce universally convergent moving averages for a fixed 
non-zero function $f$.  Given any non-zero mean-zero function 
$f$ and ergodic measure preserving transformation $T$, there exists 
a measure preserving isomorphism $\phi$, $v_n$ and increasing $L_n$ 
such that for a.e. $x\in X$, 
\[
\limsup_{n\to \infty} \frac{1}{L_n} \sum_{i=1}^{L_n} 
f(\phi^{-1}\circ T^{v_n+i} \circ \phi(x)) > 0 . 
\]
Here is a sketch of the construction of $\phi$.  
Consider a sequence of Rokhlin towers 
for $T$, and very rarely identify points $x$ in the tower where 
$f(x) > \delta$ for some fixed $\delta$.  Call these points $A_m$ 
for a tower of height $h_m^2$.  Define $\phi_m$ to map a small fraction 
of these points to the top $h_m$ levels of the tower.  
Thus, for $v_{n+i} = i h_m$ 
and $L_{n+i}=h_m+i$ for $0\leq i \leq h_m$, 
we can obtain a bad moving average for most points $x\in X$.  
If $h_m$ grows sufficiently fast, then 
$\lim_{m\to \infty} \phi_m(x) = \phi(x)$ converges for a.e. 
$x\in X$.  

Hoeffding's inequality \cite{Hoef} may be used to prove 
a stronger result for more general sequences $L_n$.  
\begin{lem}[Hoeffding's inequality]
Let $X_1, X_2, \ldots , X_n$ be independent random variables such that 
$a_i \leq X_i \leq b_i$ almost surely.  If 
$S_n = \sum_{i=1}^{n} X_i$, then for all $t>0$, 
\[
\mu \big( | S_n - E[S_n]| \geq t \big) \leq 
2 \exp{\Big( \frac{-2t^2}{\sum_{i=1}^{n} (b_i-a_i)^2} \Big)} . 
\]
\end{lem}
Thus, if $L_n \geq n^{\alpha}$ for $\alpha > 1/2$ and $g$ is bounded, 
mean-zero with $X_i = g\circ T^i$ i.i.d., then 
\begin{align}
\mu \big( | \frac{1}{L_n} \sum_{i=0}^{L_n-1} g(T^i x)| \geq \delta \big) 
&= \mu \big( | \sum_{i=0}^{L_n-1} g(T^i x)| \geq \delta L_n \big) \\ 
&\leq 2 \exp{\Big( \frac{-2(\delta L_n)^2}{4n ||g||_{\infty}^2} \Big)} \\ 
&\leq 2 \exp{\Big( \frac{-2\delta^2 n^{2\alpha-1}}{4||g||_{\infty}^2} \Big)} . 
\end{align}
Since these terms are summable, 
applying these inequalities with $g = \mathcal{E}(f|P_n)$ 
in Theorem \ref{C-of-B} yields the more general result 
for $L_n \geq n^{\alpha}$ with $\alpha > 1/2$.  

Another remark is that Theorem \ref{C-of-B} can be generalized 
to functions $f$ which are piecewise uniformly continuous 
on a finite number of elements from some partition $P_n$.  
This will be used in the next section to prove a general 
result for any bounded measurable function $f$. 

\subsection{Bounded functions}
In this section, we obtain an existence result for any bounded 
measurable function $f$. 
\begin{thm}
\label{B-of-B}
Let $(X, \B, \mu)$ be a standard probability space and 
$f:X\to \real$ a bounded measurable function.  
There exists a Bernoulli transformation $T$ such that 
for any $\alpha > 1/2$, all moving averages $M(v_n,L_n)^T f$ 
converge pointwise a.e. to $\int_X f d\mu$ for 
$v_n\in \mathbb Z$ and $L_n \geq n^{\alpha}$. 
\end{thm}
\begin{proof}
It is sufficient to prove this theorem where 
$X=[0,1]$, $\mu$ is Lebesgue measure and 
$f:X\to \real$ is a bounded mean-zero function.  
Apply Lusin's theorem to produce a sequence $D_n^{\prime}$ 
for $n\geq 1$ of closed subsets such that 
$D_{n}^{\prime} \subseteq D_{n+1}^{\prime}$ for $n\geq 1$, 
$f$ is continuous on $D_n^{\prime}$ and 
$E_n^{\prime} = X\setminus D_n^{\prime}$ satisfies 
\[
\sum_{n=1}^{\infty} \mu(E_n^{\prime}) < \infty . 
\]
For $n\in \mathbb Z^+$, let 
\[
m_n = \min{\{ m\in \mathbb Z^+\cup \{0\}: \frac{m}{2^n} \geq \mu(E_n)\}} . 
\]
Choose an open subset $E_n^{\prime\prime} \subset D_n^{\prime}$ such 
that $\mu(E_n^{\prime\prime}) = \frac{m_n}{2^n}-\mu(E_n^{\prime})$.  
Define $E_n = E_n^{\prime}\cup E_n^{\prime\prime}$ and 
$D_n = X\setminus E_n$.  
Define $\phi_1:[0,1]\to [0,1]$ by 
\begin{eqnarray*} 
\phi_1(x) = 
\left\{\begin{array}{ll}
\mu \big( E_1 \cap [0,x)\big) \  & \mbox{if $x \in E_1$}, \\ 
\mu \big( E_1\big) +\mu \big( D_1\cap [0,x)\big) & \mbox{if $x \in D_1$}.
\end{array}
\right.
\end{eqnarray*}
The function $\phi_1$ is invertible and $\mu$-preserving. 
Define $\phi_n$ inductively as 
\begin{eqnarray*} 
\phi_n(x) = 
\left\{\begin{array}{ll}
\mu \big( E_n \cap [0,x)\big) \  & \mbox{if $x \in E_n$}, \\ 
\mu \big( E_n\big) +\mu \big( (E_{n-1}\setminus E_n) \cap [0,x)\big) & \mbox{if $x \in E_{n-1}$} \\ 
\phi_{n-1}(x) & \mbox{if $x \in D_{n-1}$} . 
\end{array}
\right.
\end{eqnarray*}
For almost every $x\in [0,1]$, 
$\phi(x) = \lim_{n\to \infty} \phi_n(x)$ exists and 
the map $\phi:[0,1]\to [0,1]$ is invertible and measure preserving.  

Let $S:[0,1]\to [0,1]$ be a Bernoulli shift with a sequence 
$k_n$ such that dyadic intervals of length $\frac{1}{2^n}$ 
are measurable with respect to the partition 
\[
\bigvee_{i=-k_n-1}^{k_n-1} S^{i} \big( [0,1/2)), S^i([1/2, 1) \big) . 
\]
Define our transformation $T$ as 
\[
T = \phi^{-1} \circ S \circ \phi . 
\]
For convenience, assume $f$ has mean-zero, or similarly let 
$f = f - \int fd\mu$.  
Given $\delta > 0$, choose $n\in \mathbb Z^+$ such that 
\[
\mu (E_n) < \frac{\delta}{4||f||_{\infty}} . 
\]
Thus, if $d_n = \int_{D_n} f d\mu$, then $d_n < \frac{\delta}{4}$.  
Hence, 
\begin{align}
\mu &\big( \{ x : \frac{1}{L_m} \sum_{i=0}^{L_m-1} f (T^i x) 
\geq \delta \} \big) \\ 
&\leq 
\mu \big( \{ x: \frac{1}{L_m} \sum_{i=0}^{L_m-1} \mathbbm{1}_{E_n} (T^i x) 
\geq \frac{\delta}{2||f||_{\infty}} \} \big) + 
\mu \big( \{ x: \frac{1}{L_m} \sum_{i=0}^{L_m-1} 
\mathbbm{1}_{D_n}(T^i x) f(T^i x) \geq \frac{\delta}{2} \} \big) \\ 
&= 
\mu \big( \{ x: \frac{1}{L_m} \sum_{i=0}^{L_m-1} 
\mathbbm{1}_{\phi(E_n)} (S^i x) 
\geq \frac{\delta}{2||f||_{\infty}} \} \big) + 
\mu \big( \{ x: \frac{1}{L_m} \sum_{i=0}^{L_m-1} 
\mathbbm{1}_{D_n}(T^i x) f(T^i x) 
\geq \frac{\delta}{2} \} \big) \\ 
&= 
\mu \big( \{ x: \frac{1}{L_m} \sum_{i=0}^{L_m-1} 
\mathbbm{1}_{\phi(E_n)} (S^i x) - \mu(E_n) 
\geq \frac{\delta}{2||f||_{\infty}} - \mu(E_n) \} \big) \label{B-of-B-eqn1} \\ 
&+ 
\mu \big( \{ x: \frac{1}{L_m} \sum_{i=0}^{L_m-1} 
\mathbbm{1}_{D_n}(T^i x) f(T^i x) - d_n
\geq \frac{\delta}{2} - d_n \} \big) \label{B-of-B-eqn2}
\end{align}
Since $X_i(x)=\mathbbm{1}_{\phi(E_n)} (S^{2ik_n+k} x)$ is i.i.d., 
then the sum obtained from the terms 
in (\ref{B-of-B-eqn1}) can be decomposed into 
$2k_n$ infinite sums which are all finite.  Thus, 
for $\alpha > 1/2$ and $L_m \geq m^{\alpha}$, 
\[
\sum_{m=1}^{\infty} \mu \big( \{ x: \frac{1}{L_m} \sum_{i=0}^{L_m-1} 
\mathbbm{1}_{\phi(E_n)} (S^i x) - \mu(E_n) 
\geq \frac{\delta}{2||f||_{\infty}} - \mu(E_n) \} \big) < \infty . 
\]
Since $f$ is continuous on the compact set $D_n$, then applying 
the same technique in Theorem \ref{B-of-B} to the terms 
in (\ref{B-of-B-eqn2}) show that 
\[
\sum_{m=1}^{\infty} \mu \big( \{ x: \frac{1}{L_m} \sum_{i=0}^{L_m-1} 
\mathbbm{1}_{D_n}(T^i x) f(T^i x) - d_n 
\geq \frac{\delta}{2} - d_n \} \big) < \infty . 
\]
Hence, for $L_m \geq m^{\alpha}$, 
\[
\sum_{m=1}^{\infty} \mu \big( \{ x : \frac{1}{L_m} 
\sum_{i=0}^{L_m-1} f (T^i x) \geq \delta \} \big) < \infty . 
\]
A similar argument shows that 
\[
\sum_{m=1}^{\infty} \mu \big( \{ x : \frac{1}{L_m} 
\sum_{i=0}^{L_m-1} f (T^i x) \leq -\delta \} \big) < \infty . 
\]
This completes our proof. 
\end{proof}

\subsection {\bf Series approach} \label{Seriesapproach}

If we have $f \in L^p(\mu), 1 < p < \infty$, we could seek to produce a map $T$ with the  property that $\sum\limits_{n=1}^\infty \|M(0,n)^Tf\|^p_p < \infty$.  This guarantees that the function is good for all moving averages with respect to $T$.  Coboundaries with transfer functions in $L^p(\mu)$ have this series property, so in cases where $f$ is a $T$-coboundary with transfer function in $L^1(\mu)$, then this would be true.  In any case, it is also possible that it would easier to prove that the series condition can hold with an appropriate choice of $T$ depending on $f$.  In some sense this is what one is showing when one can construct a map $T$ for $f\in L^p(\mu)$ such that $(f\circ T^k: k \ge 1)$ are IID, except that for terms in this series one has to pass to a subsequence $(M(0,\lfloor n^D\rfloor )f)$ with $D(p-1) > 1$.

One might ask if the stronger property holds: that it is possible for every non-zero, mean-zero function $f\in L^1(\mu)$ to admit an ergodic map $T$ such that
$\sum\limits_{n=1}^\infty \|M(0,n)^Tf\|_1 < \infty$?  But this never occurs.  Indeed, trivially,
\[\sum\limits_{n=1}^N \|M(0,n)^Tf - M(0,n+1)^T f\|_1 \le \sum\limits_{n=1}^N \|M(0,n)^Tf\|_1 + \sum\limits_{n=1}^N \|M(0,n+1)^T f\|_1.\]
Also,
\[\sum\limits_{n=1}^N \|M(0,n)^Tf - M(0,n+1)^T f\|_1 =  \sum\limits_{n=1}^N \|\frac {M(0,n)^Tf}{n+1} - \frac {f\circ T^{n+1}}{n+1}\|_1.\]
Hence,
\[\sum\limits_{n=1}^N \frac {\|f\|_1}{n+1} \le 3\sum\limits_{n=1}^{N+1}\|M(0,n)^Tf\|_1.\]
Therefore, the divergence of the harmonic series proves that $\sum\limits_{n=1}^\infty \|M(0,n)^T (f)\|_1 = \infty$ for all non-zero $f$.  

\begin{rem}
To summarize this section on IID sequences, here is what the results say about statistical averaging of IID data.  If you are averaging increasingly large blocks of terms and do not have control of where they start, then you will do best if you have each average using at least twice the amount of data you used before.  That is, one should use something like $L_n = 2^n$ terms when computing the moving average.  Then you do not need to worry about moments not being finite, as long as the necessary condition for the Law of Large Numbers is in place: $X_0 \in L^1$.
\end{rem}

\section {\bf When the function is fixed}\label{basicsdata}

Now we consider the polar case to Section~\ref{basicsmap}: we fix the mean-zero function $f\in L^1(\mu)$ and ask for results that give $f$ being good or bad as we vary the map and the structure of the moving average.

\begin{prop}\label{mostmapsleveldiv}  Fix an increasing $(L_n)$ in $\mathbb Z^+$ and fix a non-zero, mean-zero $f \in L^1(\mu)$.  Then there is some $\delta > 0$ and a dense $G_\delta$ set of ergodic maps $\mathcal B$ such that for $T\in \mathcal B$ we have $\sum\limits_{n=1}^\infty \mu \{|M(0,L_n)^T f| \ge \delta\} = \infty$.
\end{prop}
\begin{proof}  As usual, we take the maps $\mathcal T$ with the weak topology.  So with the symmetry pseudo-metric, this group is a complete, metric group.
Fix $\delta > 0$.  Consider all $T$ such that the series $\sum\limits_{n=1}^\infty \mu \{|M(0,L_n)^T f| \ge \delta\}$ converges.  That is, consider
\[\mathcal B = \bigcup\limits_{K=1}^\infty \bigcap_{N=1}^\infty \{T: \sum\limits_{n=1}^N \mu \{|M(0,L_n)^T f| \ge \delta\} \le K\}.\]
It is easy to show that in the weak topology, $\{T: \sum\limits_{n=1}^N \mu \{|M(0,L_n)^T f| \ge \delta\}\le K \}$ is a closed set.  So the intersection  $\mathcal B_K = \bigcap_{N=1}^\infty \{T: \sum\limits_{n=1}^N \mu \{|M(0,L_n)^T f| \ge \delta\} \le K\}$ is also a closed set in the weak topology.  Hence, $\mathcal B$ is  an $F_\sigma$ set in the weak topology.

We claim that for a suitable $\delta$, the intersection $\mathcal B_K$ has no interior, and so then the complement of $\mathcal B$ is a dense $G_\delta$ consisting of maps with the property that we want.  Since the ergodic maps are also a dense $G_\delta$ set in $\mathcal T$, the intersection with $\mathcal B$ gives the dense $G_\delta$ set in this result.

Here is an argument that shows this intersection cannot have interior for a suitable $\delta > 0$.  Take $\delta > 0$ such that $\mu \{|f| \ge \delta\} > 0$.  Let $E = \{|f| \ge \delta\}$.  Take any map $T$ and a weak neighborhood $N$ of $T$.  For a subset $E_0$ of $E$ of sufficiently small, but positive, measure, we can construct a map $T_0$ in $N$ which is the identity on $E_0$.  Then for any $n$, $\mu \{|M(0,L_n)^{T_0}f| \ge \delta\} \ge \mu(E_0) > 0$. Hence $T_0$ is not in the intersection with that choice of $\delta$.
\end{proof}

\begin{cor}\label{mostmapsnormdiv} Let $1 \le p < \infty$.  Fix a non-zero, mean-zero $f \in L^p(\mu)$.  Then for a dense $G_\delta$ set of ergodic maps $T$, we have
\[\sum\limits_{n=1}^\infty \|M(0,n)^Tf\|_p^p = \infty.\]
\end{cor}
\begin{proof}  We simply use the inequality $\delta^p \mu \{|M(0,n)^Tf|\ge \delta \} \le \|M(0,n)^Tf\|_p^p$, and Proposition~\ref{mostmapsleveldiv}.
\end{proof}

\begin{rem}   This result is in contrast with the examples in Proposition~\ref{normseriesconv}, given by Derriennic and Lin~\cite{DL}.
\qed \end{rem}

\begin{prop}\label{mostmapsnot}  Fix an increasing $(L_n)$ in $\mathbb Z^+$ and a non-zero, mean-zero $f \in L^1(\mu)$.  Then there is a dense $G_\delta$ of ergodic maps $T$ such that for each $T$ there is some moving average with lengths $(L_n)$ for which $f$ is bad.
\end{prop}
\begin{proof}  We use the same argument as in Proposition~\ref{mostmapsleveldiv} to fix $\delta$ and produce the class of maps for which each has
\[\sum\limits_{n=1}^\infty \mu \{|M(0,n)^Tf| \ge \delta\} = \infty.\]
Then we apply Lemma~\ref{Div} to get moving averages with $\limsup\limits_{n\to \infty} |M(v_n,L_n)^Tf| \ge \delta$ a.e.  But the norm convergence shows at the same time that $\liminf\limits_{n\to \infty} |M(v_n,L_n)^Tf| = 0$ a.e.  Hence $f$ is bad for this moving average.
\end{proof}

More is true if the function is not bounded.

\begin{prop}  Fix an increasing $(L_n)$ in $\mathbb Z^+$.  If $f\in L^1(\mu) \backslash L^\infty(\mu)$, there is a dense $G_\delta$ set $\mathcal G$ of maps such that for any $K$, if $T \in \mathcal G$, we have $\sum\limits_{n=1}^\infty \mu \{|M(0,L_n)^T f| \ge K \} = \infty$.   Hence, for any map $T \in \mathcal G$, there is a moving average with $\limsup\limits_{n\to \infty} |M(v_n,L_n)^Tf| = \infty$ a.e. while as usual $\liminf\limits_{n\to \infty} |M(v_n,L_n)^Tf| = 0$ a.e.
\end{prop}
\begin{proof}  For each $K$, there is a set $E_K$ of positive measure on which $|f| \ge K$.  We use these sets as in the proof of Proposition~\ref{mostmapsleveldiv} to show that there for each $K$ there is a dense $G_\delta$ set $\mathcal G_K$ of maps with $\sum\limits_{n=1}^\infty \mu \{|M(0,L_n)^T f| \ge K \} = \infty$.  Let $\mathcal G = \bigcap\limits_{K=1}^\infty \mathcal G_K$.  Now if $T \in \mathcal G$, we can construct the moving averages in disjoint blocks $B_K$ of whole numbers and translates $v_n$ for $n\in B_K$ so that $\sup_{n \in B_K} |M(v_n,L_n)^Tf| \ge K$ on a set of measure at least $1 - \frac 1{2^n}$.   This would give the moving averages with the desired divergence property.
\end{proof}

\begin{rem}   In any case, here is what the results above say about the good functions.  Fix an ergodic map and consider all moving averages that fail the Cone Condition.  Each such bad moving average has only a first category set of good functions associated with it, but still there is a dense class of functions that are good for all these moving averages.  However, if you vary the maps AND the moving averages, there is no non-zero, mean-zero function that can be a good functions for all of these cases.
\qed \end{rem}

\section{\bf Approximation by Coboundaries}\label{approxmethod}

We can also construct good functions for all moving averages with respect to $T$ in terms of how well they are approximated by coboundaries.  First, we generally know that  every function is a series of coboundaries, although this representation is far from uniquely determined.

\begin{prop}\label{allCOBseries} Fix $r$ and $s$ with $1 \le r\le s < \infty$.  For any mean-zero function $f\in L^r(\mu)$, there is a sequence $(h_k)$ in $L^s(\mu)$ such that the series  $\sum\limits_{k=1}^\infty h_k - h_k\circ T$ converges in $L^r$-norm and 
$f = \sum\limits_{k=1}^\infty h_k - h_k\circ T$.
\end{prop}
\begin{proof} This follows directly because the ergodicity of $T$ shows that for any $\epsilon > 0$, there exists $h \in L^s(\mu)$ with $\|f - (h - h\circ T)\|_r \le \epsilon$. 
\end{proof}

However, if $f$ is not a $T$-coboundary and one approximates $f$ in $L^1$-norm by a coboundary $h - h\circ T$ with $h \in L^1(\mu)$, then the better the approximation, then one must allow that the norm $\|h\|_1$ becomes larger.  But we can give examples where we can control the rate of approximating to $f$ by coboundaries, with transfer functions not growing in norm too fast, and get functions good for all $T$ for moving averages. This is easier in $L^p(\mu), 1 < p < \infty$, and harder in $L^1(\mu)$.   

We can formalize this balance between the approximation by coboundaries and the norm of the transfer function, and connect it with the rate of convergence to zero of the ergodic averages.  Here are the two relevant results, which are surely well-known. We have this rate of convergence over estimate.

\begin{prop}\label{rateover}
 Take $f \in L^r(\mu)$
and $H \in L^r(\mu)$.  Then
\[\|A_n^T f\|_r \le \|A_n^T(f - (H-H\circ T))\|_r + 2\|H\|_r/n.\]
So
\[\|A_n^T f\|_r\le \inf_{\|H\|_r\le L}\|f - (H-H\circ T)\|_r + 2L/n.\]
\end{prop}

\noindent Hence, with $T$ ergodic, the approximation of mean-zero $f$ by
coboundaries bounds the rate that $\|A_n^T f\|_r$ goes to zero.  But in addition,  there is this rate of convergence under estimate.

\begin{prop}\label{rateunder} Given $f\in L^r(\mu)$, we can write $f - A_n^T f$ as the coboundary $H - H \circ T$ with $H = \frac
1n\sum\limits_{k=1}^n S_{k-1}^T f$.  So $\|H\|_r\le n\|f\|_r$ and
\[ \|A_n^T f\|_r \ge  \inf_{\|H\|_r\le n\|f\|_r}\|f - (H-H\circ T)\|_r.\]
\end{prop}

\noindent So, with $T$ ergodic, the rate that $\|A_n^T f\|_r$ goes to zero bounds the rate of
approximation of $f$ by coboundaries.

\begin{rem}\label{rates}
The well-known basic theorem is that for any $\epsilon_n \to 0$, no matter how slowly, and any sequence $(m_n)$ and any $r, 1 \le r \le \infty$, there is a dense
$G_\delta$ set of functions $f\in L^r(\mu)$ such that
\[\limsup\limits_{n\to \infty}(1/ \epsilon_n)\|A_{m_n}^Tf\|_r = \infty.\]
It is a good question to ask for which functions one has actually
\[\lim\limits_{n\to \infty} (1/ \epsilon_n)\|A_{m_n}^Tf\|_r = \infty.\]
Of course, one can also ask for the nature of the functions $f$ such that
\[\liminf\limits_{n\to \infty} (1/\epsilon_n)\|A_{m_n}^Tf\|_r = 0 \quad \text {or}\quad \lim\limits_{n\to \infty} (1/\epsilon_n)\|A_{m_n}^Tf\|_r = 0.\]
If $f$ is a coboundary with transfer function in $L^1(\mu)$ and $1/\epsilon_n =
o(m_n)$, then these both hold.  See Rosenblatt~\cite{Rosenblatt} for
many results related to the (optimal) speed of norm convergence.
\qed \end{rem}

\subsection{\bf Some constructions}\label{EGS}

Suppose we have $f\in L^1(\mu)$ and a sequence $(g_n)$ of functions in $L^2(\mu)$.  Suppose we have two conditions:
\begin{equation}\label{one} 
\sum\limits_{n=1}^\infty \|f - (g_n - g_n\circ T)\|_1 < \infty.
\end{equation}
and also  
\begin{equation}\label{two} 
\sum\limits_{n=1}^\infty \|g_n\|_2^2/n^2 < \infty.
\end{equation}

At least formally, one hopes that Equation~\ref{two} allows sufficient growth in the norms $\|g_n\|_1$ such that the degree of the approximation in Equation~\ref{one} is possible.  Below will be some examples that exhibit proof of concept for these two equations to hold simultaneously.

\begin{prop}\label{closest} If $f\in L^1(\mu)$ and there is a sequence $(g_n)$ of functions in $L^2(\mu)$ such that Equation~\ref{one} and Equation~\ref{two} hold, then $f$ is good for all moving averages with respect to $T$.
\end{prop}
\begin{proof}  We need to show that for all $\delta > 0$, the series
$\sum\limits_{n=1}^\infty \mu(\{|A_n^Tf| \ge \delta \}) < \infty$.  It suffices to show that the following two series converge:

\[\sum\limits_{n=1}^\infty \mu(\{|A_n^T (f - (g_n -g_n\circ T)|\ge \delta/2\})\]
\[\sum\limits_{n=1}^\infty \mu(\{|A_n^T (g_n - g_n\circ T)|\ge \delta/2\})\]

The first series is bounded by $\sum\limits_{n=1}^\infty (2/\delta)
\|f - (g_n - g_n\circ T)\|_1$, which is convergent by Equation~\ref{one}.
The second series also converges.  First, $A_n^T ( g_n - g_n\circ T) = 
\frac 1n (g_n\circ T - g_n \circ T^{n+1})$.  So, the usual Chebychev inequality, $\mu(\{|A_n^T (g_n - g_n\circ T)|\ge \delta/2\}) \le
(16/\delta)\|g_n\|_2^2$.  Hence, Equation~\ref{two} shows that
$\sum\limits_{n=1}^\infty \mu(\{|A_n^T (g_n - g_n\circ T)|\ge \delta/2\})$ converges.  
\end{proof}

We can get many variations on this result by replacing $L^1(\mu)$ by $L^r(\mu)$ for some $r \ge 1$, and $L^2(\mu)$ by $L^s(\mu)$ with $s\ge r$.  For example, we have
the following result that uses this equation:

\begin{equation}\label{three}
\sum\limits_{n=1}^\infty \|f - (g_n - g_n\circ T)\|_2^2 < \infty.
\end{equation}

\begin{prop}\label{closer} If $f\in L^1(\mu)$ and there is a sequence $(g_n)$ of functions in $L^2(\mu)$ such that Equation~\ref{two} and Equation~\ref{three} hold, then $f$ is good for all \ averages.
\end{prop}
\begin{proof}  The argument is similar except that we show that
$\sum\limits_{n=1}^\infty \mu(\{|A_n^T (f - (g_n -g_n\circ T)|\ge \delta/2\})$ converges by over-estimating it by 
$\sum\limits_{n=1}^\infty 
(4/\delta^2)\|A_n^T (f - (g_n -g_n\circ T)\|_2^2$.
\end{proof}

We can create functions $f$ which satisfying conditions such as in 
Proposition~\ref{closest} or Proposition~\ref{closer} by taking series of coboundaries.  We assume at least that $(h_k)$ is a sequence in $L^1(\mu)$ 
and $f = \sum\limits_{k=1}^\infty h_k - h_k\circ T$ converges in $L^1$-norm.   We then write $f = g_n - g_n\circ T + t_n$ where
$g_n = \sum\limits_{k=1}^n h_k$ and $t_n = \sum\limits_{k=n+1}^\infty h_k - h_k\circ T$.  If we assume that $h_k \in L^2(\mu)$ for all $k$, then
Proposition~\ref{closest} and Proposition~\ref{closer} can give results of which the following is just one simple version.

\begin{prop}\label{seriesbasic} Suppose $(h_k)$ is a sequence in $L^2(\mu)$
such that 
\[\|h_k\|_2 = O(1/k^{3/4})\] 
and 
\[\sum\limits_{k=1}^\infty k\|h_k - h_k\circ T\|_1 < \infty.\]
Then the series $f = \sum\limits_{k=1}^\infty h_k - h_k\circ T$ gives a 
$L^1$-norm convergent series such that $f$ is good for all moving averages with respect to $T$.  
\end{prop}
\begin{proof}  With $g_n = \sum\limits_{k=1}^n h_k$, we have
$\|g_n\|_2 = O(n^{1/4})$, so Equation~\ref{two} holds.  But also we have
$\sum\limits_{n=1}^\infty \|f - (g_n - g_n\circ T)\|_1 \le
\sum\limits_{n=1}^\infty \|t_n\|_1 = \sum\limits_{n=1}^\infty 
\|\sum\limits_{k=n+1}^\infty h_k - h_k\circ T\|_1$.  So Equation~\ref{one} holds.  That is,  $\sum\limits_{n=1}^\infty \|f - (g_n - g_n\circ T)\|_1  <\infty $.  This is because \[\sum\limits_{n=1}^\infty 
\|\sum\limits_{k=n+1}^\infty h_k - h_k\circ T\|_1
\le \sum\limits_{k=2}^\infty (k-1)\|h_k - h_k \circ T\|_1\]
and we are assuming $\sum\limits_{k=1}^\infty k\|h_k - h_k\circ T\|_1 < \infty$.
\end{proof}

Clearly there are many improved versions of the estimates in Proposition~\ref{seriesbasic} which would give a wider class of examples of series $f= \sum\limits_{k=1}^\infty h_k - h_k\circ T$ which are convergent and give functions $f$ that are good for all moving averages with respect to $T$.

For this type of construction to be really useful, we would want to know that we can also arrange that the series $f=  \sum\limits_{k=1}^\infty h_k - h_k\circ T$ is not a $T$-coboundary itself.  However, it is not very clear how to guarantee that $f$ is not a $T$-coboundary in terms of $(h_k)$.  Indeed, there is this companaion result to Proposition~\ref{allCOBseries}.

\begin{prop}\label{divCOBseries}  Given any mean-zero $f\in L^1(\mu)$, there is a sequence $(h_k)$ in $L^1(\mu)$ such that $\sum\limits_{k=1}^\infty h_k$ does not converge in $L^1$-norm and yet $ \sum\limits_{k=1}^\infty h_k - h_k\circ T$ converges in $L^1$-norm, and $f = \sum\limits_{k=1}^\infty h_k - h_k\circ T$. 
\end{prop}
\begin{proof}   Suppose we have inductively chosen 
$(h_k: 1\le k \le K)$ with $\|h_k\|_1 \ge k$ for all $k, 1 \le k \le K$, and such that
$\|f - \sum\limits_{k=1}^K h_k - h_k\circ T\|_1 \le 1/2^K$.  Choose $g\in L^1(\mu)$ that is not a $T$-coboundary with $\|f - (\sum\limits_{k=1}^K h_k - h_k\circ T) - g\|_1 \le \frac 12(1/2^{K+1})$.
Then because $g$ is not a $T$-coboundary, 
we can choose $h_{K+1}\in L^1(\mu)$ with $\|h_{K+1}\|_1 \ge K+1$ such that 
$\|g - (h_{K+1}- h_{K+1}\circ T)\|_1 \le \frac 12(1/2^{K+1})$.  This gives
$\|f - \sum\limits_{k=1}^{K+1} h_k - h_k\circ T\|_1 \le 1/2^{K+1}$ and $\|h_k\|_1 \ge k$ for all
$k=1,\dots,K+1$.  This induction gives a sequence $(h_k)$ in $L^1(\mu)$ such that the series
$\sum\limits_{k=1}^\infty h_k - h_k\circ T$ converges to $f$ in $L^1$-norm.  But since $\|h_k\|_1$ tends to
infinity, the series $\sum\limits_{k=1}^\infty h_k$ cannot converge in $L^1$-norm.
\end{proof}

In Proposition~\ref{divCOBseries}, it does not matter if $f$ is a $T$-coboundary or not.  So this result shows that just because $\sum\limits_{k=1}^\infty h_k$ is not norm convergent, while
$f = \sum\limits_{k=1}^\infty h_k - h_k\circ T$ is norm convergent, we cannot conclude that $f$ is not a
$T$-coboundary with transfer function in $L^1(\mu)$.  In the following example, we work around this issue 
to give such a construction with an appropriate choice of $(h_k)$ for which indeed $f$ is not a $T$-coboundary.  
\medskip

\noindent {\bf Example}:  Here is an example of a series of coboundaries that by the results above is good for moving averages with respect to $T$, but  is not itself a $T$-coboundary with a transfer function in $L^1(\mu)$.  We take $a_k =1/k^{1/2}$ for all $k$.  We take also measurable sets $E_k$ with $\mu(E_k) = 1/k^{1/2}$ that are $T$ almost invariant to the degree that
$\sum\limits_{k=1}^\infty k^{1/2} \mu(E_k\Delta T (E_k) < \infty$.
Let $h_k = a_k 1_{E_k}$.  Then
consider $f = \sum\limits_{k=1}^\infty h_k - h_k\circ T$.
We have $\sum\limits_{k=1}^\infty \|h_k - h_k\circ T\|_1 \le
\sum\limits_{k=1}^\infty (1/k^{1/2}) \mu(E_k\Delta T(E_k))$.  So 
$f$ is an $L^1$-norm convergent series.   We also have 
$\|h_k\|_2 = a_k\mu(E_k)^{1/2} = (1/k^{1/2}) (1/k^{1/4}) = 1/k^{3/4}$.  This gives the first hypothesis in Proposition~\ref{seriesbasic}.  Also,
\[\sum\limits_{k=1}^\infty k\|h_k - h_k\circ T\|_1 \le
\sum\limits_{k=1}^\infty (k/k^{1/2}) \mu(E_k\Delta T(E_k)) < \infty.\]
This gives the second hypothesis in Proposition ~\ref{seriesbasic}.  Hence,
$f$ is good for all moving averages with respect to $T$. 

We want to arrange now for $f$ to not be a $T$-coboundary.  For technical reasons, we rewrite
$f = \sum\limits_{k=1}^\infty h_k^0 - h_k^0\circ T$ where $h_k^0 = h_k -\int_X h_k \, d\mu$.  In our case,
we have $\|h_k^0\|_1 = (2/k)(1 - 1/k^{1/2})$.  Hence, we have $\sum\limits_{k=1}^\infty \|h_k^0\|_1 = \infty$.  This in itself is not enough to show that $f$ is not a $T$-coboundary; see Proposition~\ref{divCOBseries}.  So we make these additional inductive adjustments.

Let $S_n^T (f) = \sum\limits_{k=1}^n f\circ T^k$.  Let
$C_N^T (f) =\frac 1N \sum\limits_{n=1}^N S_n^T (f)$.  It is well-known that $f$ is a $T$-coboundary with transfer function $F\in L^1(\mu)$ if and only if $(C_N^T (f))$ converges in $L^1$-norm.  Moreover, if $F$ is mean-zero, then $C_N^T (f) \to F$ in $L^1$-norm as $N\to \infty$.   See Lin and Sine~\cite{LS}.
\medskip

Now suppose we have  chosen $(h_k:k=1,\dots,K)$ as above.   Then
for $N_K$ large enough, we would have $\|C_{N_K}^T (\sum\limits_{k=1}^K h^0_k - h^0_k\circ T)\|_1 \ge \frac 12\|\sum\limits_{k=1}^K h^0_k\|_1$.  
We then need to choose the $h_k, k > K$ such that we still have
$\|C_{N_K}^T (f)\|_1 \ge \frac 14\|\sum\limits_{k=1}^K h^0_k \|_1$.   This would be easier to achieve, and in fact the additional constraints on $(h_k)$ would not even be necessary, if $C_n$ were uniformly bounded operators.  However, it is easy to see that on the mean-zero functions in $L^1(\mu)$, the operator norm $\|C_N\|_1 = \frac {N+1}2$.  To deal with this, we just make sure that the $E_k,k > K$ are sufficiently $T$-invariant, so that in turn $\|h_k - h_k\circ\tau\|_1$ is sufficiently small, so that
$\|C_{N_K}^T (\sum\limits_{k=K+1}^\infty h^0_k - h^0_k\circ \tau)\|_1 \le \frac 14\|\sum\limits_{k=1}^K h^0_k \|_1$.   This is easy to arrange inductively without affecting the conditions already imposed above.  This then implies that $C_N^T (f)$ does not converge in $L^1(\mu)$.  Hence, $f$ is not a $T$-coboundary with transfer function in $L^1(\mu)$.

It is likely that the series of coboundaries construction given here can give functions that are not in any known class of functions 
good for all moving averages with respect to $T$.  For example, we should be able to arrange that $f$ is also not a fractional coboundaries as discussed in Derriennic and Lin~\cite{DL}.

\section{Functions with No Universal Moving Averages}

In this section, we give conditions on mean-zero functions 
such that for any invertible ergodic measure preserving 
transformation $T$, there is a divergent moving average. 
First, we define a specific family of mean-zero functions. 

\subsection{Definition of $\mathcal{F}$}
Let $a_n >2$ be a sequence of real numbers 
with super polynomial growth.  Choose disjoint subsets 
$A_n \subset X$ of positive measure such that 
\[
\sum_{n=1}^{\infty} a_n \mu (A_n) \leq \frac{1}{2} . 
\]
Choose a subset $B \subset X$ 
such that $B\cap ( \bigcup_{n=1}^{\infty} A_n ) = \emptyset$ 
and $\mu (B) = \sum_{n=1}^{\infty} a_n \mu(A_n)$.  
Define the function $f$ as 
\begin{eqnarray*} 
f(x)= 
\left\{\begin{array}{ll}
a_n  & \mbox{if $x \in A_n$}, \\ 
- 1 & \mbox{if $x \in B$}.
\end{array}
\right.
\end{eqnarray*}
Thus, 
\[
\int_X f d\mu = \sum_{n=1}^{\infty} a_n \mu (A_n) - \mu (B) = 0. 
\]
Let $\mathcal{F}$ represent this collection of measurable functions.  
Given a real number $p>1$, let 
\[
\mathcal{F}_p = \{ f\in \mathcal{F}: \lim_{n\to \infty} a_n^p \mu(A_n)=\infty \} . 
\]
Note that $\mathcal{F}_p \cap L^p(\mu) = \emptyset$.  However, 
it is straightforward to show that 
\[
\mathcal{F}_p \cap \Big( \bigcap_{q<p} L^q(\mu) \Big) \neq \emptyset . 
\]
On the flip side, we have that 
\[
L^p(\mu) \cap \Big( \bigcap_{q>p} \mathcal{F}_q \Big) \neq \emptyset . 
\]
These properties will be useful for constructing counterexamples based 
on our main result. 
\subsection{Main results}
Our main result in this section is Theorem \ref{thm-bad-fun} which is a counterpart 
to Proposition \ref{Lp2}. 
\begin{thm}
\label{thm-bad-fun}
Let $(X, \mu, \B)$ be a standard probability space 
and $p > 1$.  Given a mean-zero function $f\in \mathcal{F}_p$ and 
any invertible ergodic $\mu$-preserving transformation $T:X\to X$, 
there exists a moving average $(v_n, L_n)$ with $L_n\geq n^{\frac{1}{p-1}}$ 
such that for a.e. $x\in X$, 
\[
\limsup_{n\to \infty} \frac{1}{L_n} \sum_{i=v_n+1}^{v_n+L_n} 
f(T^i x) > 0 . 
\]
\end{thm}
This theorem implies the following corollaries which were longstanding 
open problems on the theory of moving averages. 
\begin{cor}
There exists a mean-zero function $f$ such that $f \in L^{p}(\mu)$ for all $p<2$ 
and given any invertible ergodic $\mu$-preserving transformation $T$, 
there exists a moving average $(v_n, L_n)$ with $L_n\geq n$ such that 
for a.e. $x\in X$, 
\[
\limsup_{n\to \infty} \frac{1}{L_n} \sum_{i=v_n+1}^{v_n+L_n} f(T^i x) > 0 . 
\]
In other words, $f$ is a universally bad function with no universal convergence 
of moving averages for any invertible ergodic measure preserving transformation. 
\end{cor}
\begin{proof}
Given the sequence $a_n$ with super polynomial growth, choose disjoint subsets 
$A_n \subset X$ such that $\mu(A_n) = {n}/{a_n^2}$.  
Since $a_n^2 \mu(A_n)\to \infty$ as $n\to \infty$, then the corresponding 
function $f\in \mathcal{F}$ is also contained in $\mathcal{F}_2$.  
Thus, the function $f$ satisfies the conclusion of Theorem \ref{thm-bad-fun} 
for $p=2$ or alternatively, $L_n\geq n^{\frac{1}{p-1}}=n$.  
Since $a_n$ has super polynomial growth, then for $p<2$, 
\[
\int_{X} |f|^p d\mu = \sum_{n=1}^{\infty} a_n^p \mu(A_n) + \mu (B) 
= \sum_{n=1}^{\infty} \frac{n}{a_n^{2-p}} + \mu (B) < \infty .\ \ \Box
\]
\end{proof}

\begin{cor}
There exists a mean-zero function $f\in L^1(\mu)$ such that 
given any invertible ergodic $\mu$-preserving transformation $T$ and 
any polynomial rate $n^d$, $d\in \mathbb{Z}^+$, 
there exist integers $v_n=v_{n,d}$ such that for a.e. $x\in X$, 
\[
\limsup_{n\to \infty} \frac{1}{n^d} \sum_{i=v_n+1}^{v_n+n^d} 
f(T^i x) > 0 . 
\]
In other words, there exist integrable functions which are universally 
bad for any polynomial rate and any invertible ergodic measure preserving 
transformation. 
\end{cor}
\begin{proof}
Given the sequence $a_n$ with super polynomial growth, choose disjoint subsets 
$A_n \subset X$ such that 
\[
\mu(A_n) = \frac{3}{a_n n^2 \pi^2} . 
\]
Since 
\[
\sum_{n=1}^{\infty} \frac{1}{n^2} = \frac{\pi^2}{6} , 
\]
then the corresponding function $f$ is integrable and contained in $\mathcal{F}$.  
Since $a_n$ has super polynomial growth, then for $p>1$, 
\[
a_n^p \mu(A_n) = \frac{3a_n^{p-1}}{n^2 \pi^2} \to \infty \ \ 
\mbox{as $n\to \infty$} . 
\]
Therefore, the function $f\in \mathcal{F}_p$ satisfies the conditions 
of Theorem \ref{thm-bad-fun} for any $p>1$.  Letting 
$\frac{1}{p-1} \to \infty$ as $p\to 1^+$, proves the corollary.  $\Box$ 
\end{proof}

\subsection{Main lemma}
The previous theorem will use the following lemma. 
\begin{lem}
\label{lem-bad-fun}
Let $p > 1$ be a real number, $f \in \mathcal{F}_p$, 
and $T$ be an invertible ergodic measure preserving 
transformation on a standard probability space $(X, \mu, \B)$.  
Suppose there exists a real number $\eta$ such that $0<\eta<1$ and for 
each $n\in \mathbb Z^+$, there exists a measurable set $B_n$ such that 
$\mu (B_n) > \eta a_n \mu(A_n)$ and for each $x\in B_n$, there exists 
$\ell_x > \eta a_n$ such that 
\[
\sum_{i=0}^{\ell_x-1} f ( T^i x ) > \eta \ell_x . 
\]
Then there exists a moving average $(v_n, L_n)$ with $L_n \geq n^{\frac{1}{p-1}}$ 
such that for almost every $x\in X$, 
\[
\limsup_{n\to \infty} \frac{1}{L_n} \sum_{i=v_n+1}^{v_n+L_n} 
f ( T^{i} x ) > 0 . 
\]
\end{lem}
\begin{proof}
Choose $K > \eta a_n$ such that the set 
\begin{align}
B_n^{\prime} &= \{ x\in B_n: \ell_x \leq K \} \label{B1} 
\end{align}
satisfies 
\begin{align}
\mu (B_n^{\prime}) &> \frac{1}{2} \mu (B_n) > \frac{\eta a_n \mu (A_n)}{2} . 
\label{B2}
\end{align}
For convenience, set 
\[
\delta =  \Big( \frac{\eta^2}{16} \Big)^{\frac{1}{p-1}} . 
\]
Quantize the set of $\ell_x$ for $x\in B_n^{\prime}$ in the following manner. 
For each $x\in B_n^{\prime}$, let 
\[
\beta (x) = \min{\{ r\in \mathbb Z^+: a_n \delta \cdot r^{\frac{1}{p-1}} 
\geq \ell_x \}} . 
\]
Enumerate $r_1 < r_2 < \ldots < r_t$ such that 
\[
B_n^{\prime} = \bigcup_{i=1}^{t} \{ x\in B_n^{\prime}: \beta(x)=r_i \} . 
\]
Set 
\[
\tau = \lfloor \big( a_n \delta \big)^{p-1} \rfloor . 
\]
For $0\leq k < \tau$ and $1\leq j \leq t$, define 
\[
L_{j\tau - k} = 
\lceil a_n \delta \cdot r_j^{\frac{1}{p-1}} \rceil 
+ \lceil \frac{\eta^2 a_n^{p-1}}{8} \rceil - k . 
\] 
It is straightforward to show that 
\[
L_{j\tau - k} \geq \big( r_j \tau \big)^{\frac{1}{p-1}} \geq 
\big( j\tau \big)^{\frac{1}{p-1}} \geq 
\big( j\tau - k \big)^{\frac{1}{p-1}} . 
\]
For $x\in B_n^{\prime}$ such that $\beta(x)=r_j$, we set $k=0$. 
Thus,
\begin{align}
\sum_{i=0}^{L_{j\tau}-1} f(T^i x) &= \sum_{i=0}^{\ell_x-1} f(T^i x) + 
\sum_{i=\ell_x}^{L_{j\tau}-1} f(T^i x) \\ 
&> \eta \ell_x - \frac{\eta^2 a_n}{4} . 
\end{align}
Hence, 
\begin{align}
\frac{1}{L_{j\tau}} \sum_{i=0}^{L_{j\tau}-1} f(T^i x) &> 
\eta \frac{\ell_x}{L_{j\tau}} - \frac{\eta^2 a_n}{4\ell_x} \\ 
&> \eta \Big( \frac{\ell_x}{\ell_x+\frac{\eta^2 a_n}{4}}\Big) - 
\frac{\eta}{4} \\ 
&= \eta \Big( \frac{\ell_x+\frac{\eta^2 a_n}{4}}{\ell_x+\frac{\eta^2 a_n}{4}}\Big) - \eta \Big( \frac{\frac{\eta^2 a_n}{4}}{\ell_x+\frac{\eta^2 a_n}{4}}\Big) - \frac{\eta}{4} \\ 
&> \eta - \frac{\eta^3 a_n}{4\ell_x} - \frac{\eta}{4} > \frac{\eta}{2} . 
\end{align}
This previous equation identifies points with bad averages.  
We will use the ergodicity of $T$ to iteratively identify more points 
with a bad moving average.  
Let $D_1 = X\setminus B_n^{\prime}$.  Since $T$ is ergodic, there 
exists $m_1$ such that 
$\mu (D_1\cap T^{-m_1}B_n^{\prime}) > \frac{1}{2} \mu(D_1)\mu(B_n^{\prime})$.  
Let $D_2=D_1\setminus T^{-m_1}B_n^{\prime}$.  Thus, 
\[
\mu(D_2) < \mu(D_1) \Big( 1-\frac{1}{2}\mu(B_n^{\prime})\Big) 
< \Big( 1 - \frac{1}{2} \mu(B_n^{\prime})\Big)^2 . 
\]
Continue this procedure $\tau - 1$ times to produce sets 
$D_1, D_2, \ldots , D_{\tau-1}$ and iterates 
$m_0, m_1, \ldots , m_{\tau-1}$.  Set $m_0=0$ and let 
$\mathcal{M} = \{m_0, m_1, \ldots , m_{\tau-1} \}$.  
The set $D_{\tau-1}$ satisfies: 
\begin{align}
\mu(D_{\tau-1}) &< \Big( 1 - \frac{1}{2}\mu(B_n^{\prime})\Big)^{\tau-1} \label{A1} \\ 
&\leq \Big( 1 - \frac{\eta}{4}a_n \mu(A_n)\Big)^{(a_n\delta)^{p-1}-1}\ \ \mbox{(for sufficiently large $n$)} \\ 
&\leq 2\Big( \big( 1 - \frac{\eta}{4}a_n \mu(A_n)\big)^{\frac{1}{\eta a_n \mu(A_n)}} \Big)^{\eta \delta^{p-1} a_n^{p} \mu(A_n)} \to 0\ \mbox{as $n\to \infty$}. 
\end{align}
The previous limit equals 0, since the following limits hold: 
\begin{align}
\lim_{n\to \infty} a_n \mu(A_n) &= 0 , \\ 
\lim_{n\to \infty} \big( 1 - \frac{\eta}{4}a_n \mu(A_n)\big)^{\frac{1}{\eta a_n \mu(A_n)}} &= e^{{-1}/{4}} \\ 
\lim_{n\to \infty} a_n^p \mu(A_n) &= \infty. 
\end{align}

Define moving averages $v_{j\tau-k} = m_k$ and 
$L_{j\tau-k}=L_{j\tau}-k$ as above for $1\leq j \leq t$ 
and $0\leq k < \tau$.  
Now we show these moving averages do not converge for points in $D_{\tau}^{c}$.  
Let $x\notin D_{\tau}$.  There exists $k \leq t$ such that $x \in T^{-m_k}B_n^{\prime}$.  
Suppose $r_j = \beta(T^{m_k}x)$ and let $y=T^{m_k}x$.  
Thus, 
\begin{align}
\frac{1}{L_{j\tau-k}} \sum_{i=m_k}^{m_k+L_{j\tau-k}} f(T^i x) &= \frac{1}{L_{j\tau-k}} \sum_{i=0}^{L_{j\tau-k}} f(T^i y) \\ 
&= \frac{1}{L_{j\tau-k}} \sum_{i=0}^{\ell_x-1} f(T^i y) - \frac{1}{L_{j\tau-k}} \sum_{i=\ell_x}^{L_{j\tau-k}} f(T^i y) \\ 
&> \frac{\ell_x}{L_{j\tau-k}}\Big( \frac{\eta}{2} \Big) - \frac{k}{L_{j\tau-k}} \\ 
&> \frac{\eta}{4} - \frac{\eta^2 a_n}{8\eta a_n} = \frac{\eta}{8} . \label{A2}
\end{align}

Due to the quantization step, for $p\leq 2$, $L_{i}$ is increasing in $i$.  
Given $n_1 \in \mathbb Z^+$, we have defined a finite set of moving averages: 
\[
V_1 = \{ (m_k, L_{j\tau-k}): x\in B_{n_1}^{\prime}, m_k \in \mathcal{M}_1 \} . 
\]
Choose $n_2 \in \mathbb Z^+$ such that 
\[
\eta a_{n_2} > L_{t\tau} . 
\]
which is finite by (\ref{B1}).  
Repeat the same procedure for $n_2$ to produce a finite set of averages $V_2$.  Continue this process to produce 
an infinite collection $V_j$ for $j\in \mathbb Z^+$.  
Therefore, the lemma follows from the results of (\ref{A1})-(\ref{A2}). $\Box$
\end{proof}

\subsection{Proof of Theorem \ref{thm-bad-fun}}
Now we present a proof of the main theorem in this section. 
For $x\in X$, define 
\[
k_x = \inf{\{ k\in \mathbb Z^+: \sum_{i=0}^{k-1} f(T^{-i} x) < \frac{1}{2} k \}} . 
\]
Since $f\in L^1(\mu)$ has mean zero and $T^{-1}$ is ergodic, 
then $k_x < \infty$ for a.e. $x\in X$.  
For $x\in A_n$, $k_x > \frac{2}{3} a_n$.  
Choose $K_n \in \mathbb Z^+$ such that 
\[
\mu \big( \{ x\in A_n: k_x < K_n \} \big) > \frac{1}{2} \mu(A_n) . 
\]
Let $C = \{ I_n, T^{-1}I_n, \ldots , T^{-h_n+1}I_n \}$ be a Rokhlin tower of height $h_n$ 
such that $\frac{K_n}{h_n} < {\mu(A_n)}/{8}$ and 
\[
\mu \big( \bigcup_{i=0}^{h_n-1} T^{-i} I_n \big) > 1 - \frac{\mu(A_n)}{16} . 
\]
Let $N > n$ be such that 
\[
\mu \big( \bigcup_{i=N+1}^{\infty} A_i \big) < \frac{\mu(A_n)\mu(I_n)}{16h_n} . 
\]
Let $E = \bigcup_{i=N+1}^{\infty} A_i$ and 
\[
J_n = I_n \setminus \big( \bigcup_{i=0}^{h_n-1} T^{i}E \big) . 
\]
Thus, $\mu(J_n) > \big( 1 - \frac{\mu(A_n)}{16} \big) \mu(I_n)$.  
Hence, 
\[
\mu(\bigcup_{i=0}^{h_n-1} T^i J_n ) >  \big( 1 - \frac{\mu(A_n)}{16} \big) \big( 1 - \frac{\mu(A_n)}{16}\big) 
> 1 - \frac{\mu(A_n)}{8} . 
\]
This implies 
\[
\mu(A_n \cap \bigcup_{i=0}^{h_n-K_n} T^{-i} J_n ) > \frac{3\mu(A_n)}{4} . 
\]
If $A_n^{\prime} = \{ x\in A_n: k_x < K_n \} \cap \bigcup_{i=0}^{h_n-K_n} T^{-i} J_n$, then 
\[
\mu(A_n^{\prime}) >  \frac{\mu(A_n)}{4} . 
\]

Let $P_n$ be a finite measurable partition of $J_n$ such that for each $q\in P_n$ 
and $0\leq i < h_n$, if $\mu(A_j \cap T^{-i} q)>0$ for some $1\leq j \leq N$, then 
$\mu(A_j \cap T^{-i} q) = \mu(p)$.  Note that 
$k_x=k_y$ for $x,y \in A_n^{\prime}\cap T^{-i}q$ for $0\leq i < h_n-K_n$.  
For $x\in T^{-i}q$ and $0\leq i < h_n-K_n$, the value of $k_x$ depends on $q$ and $i$.  
Denote these values by $k_{q,i}$.  
The collection 
\[
\{ (T^{-i} q, T^{-i-1}q, \ldots , T^{-i-k_{q,i}+1}q): T^{-i} q\subseteq A_n, 0\leq i < h_n-K_n \} 
\]
forms a finite cover of $\{ A_n \cap T^{-i} q : 0\leq i < h_n-K_n \}$.  
For each $q\in P_n$, this cover can be thought of as a one-dimensional 
cover of integers by intervals of integers.  Apply the Vitali Covering Lemma \cite{Vitali} 
to extract a finite disjoint subcover for each $q\in P_n$.  Thus, for each 
$q\in P_n$, there is a finite collection $r_{q,1}, \ldots , r_{q,t_q}$ 
such that the sets 
\[
U_{q,j} = \bigcup_{i=0}^{k_{q,r_{q,j}}-1} T^{-r_{q,j}-i}q 
\]
are disjoint and 
\[
\mu \big( A_n \cap \bigcup_{q\in P_n} \bigcup_{j=1}^{t_q} U_{q,j} \big) \geq \frac{1}{3} \mu(A_n^{\prime}) > \frac{1}{12} \mu(A_n) . 
\]

Now, we bound the maximum number of iterates $T^{-i} q$ which fall in $A_n$.  By the construction, 
$T^{-r_{q,j}}q \subseteq A_n$.  
For $x\in T^{-r_{q,j}}q$, 
\[
\sum_{i=0}^{k_{q,r_{q,j}}-1} f(T^{-i} x) < \frac{1}{2} k_{q,r_{q,j}} . 
\]
To simplify this next computation, let $k = k_{q,r_{q,j}}$.  Let $\rho_1$ 
be the number of iterates that fall in $A_n$, $\rho_2$ the number 
of iterates that fall in $B$ and $\rho_3$ the number of remaining iterates.  
Thus, $\rho_1 + \rho_2 + \rho_3 = k$.  We have 
\[
\rho_1 a_n - \rho_2 + \rho_3 a < \frac{1}{2} k 
\]
where $a \geq 0$ is an average of non-negative values including $a_j$, $j\neq n$.  
This implies 
\[
\rho_1 < \frac{3k}{2a_n} .  
\]
This shows that 
\[
\mu(U_{q,j}) = \frac{k}{\rho_1} \mu(U_{q,j}\cap A_n) > \frac{2a_n}{3} \mu(U_{q,j}\cap A_n) . 
\]
Summing over all $U_{q,j}$, we get 
\begin{align}
\mu \big( \bigcup_{q\in P_n} \bigcup_{j=1}^{t_q} U_{q,j} \big) &= \sum_{q\in P_n} \sum_{j=1}^{t_q} \mu(U_{q,j}) \\ 
&> \sum_{q\in P_n} \sum_{j=1}^{t_q} \frac{2a_n}{3} \mu(U_{q,j}\cap A_n) \\ 
&= \frac{2a_n}{3} \mu \big( A_n \cap \bigcup_{q\in P_n} \bigcup_{j=1}^{t_q} U_{q,j} \big) \\ 
&> \Big( \frac{2a_n}{3} \Big) \Big( \frac{1}{12} \Big) \mu(A_n) = \frac{1}{18} a_n \mu(A_n) . 
\end{align}

Our goal is to apply Lemma \ref{lem-bad-fun}.  We're close.  
We obtain the set $B_n$ by taking subsets of $U_{q,j}$.  Start at the top of $U_{q,j}$.  
In particular, let 
\[
m_{q,j} = \lceil \frac{3k_{q,r_{q,j}}}{4} \rceil . 
\]
and let 
\[
U_{q,j}^{\prime} =  \bigcup_{i=m_{q,j}}^{k_{q,r_{q,j}}-1} T^{-r_{q,j}-i}q . 
\]
Hence, 
\[
\mu \big( \bigcup_{q\in P_n} \bigcup_{j=1}^{t_q} U_{q,j}^{\prime} \big) > \frac{1}{72} a_n \mu (A_n) . 
\]
For each $x \in U_{q,j}^{\prime}$, there exists $m$, $m_{q,j} \leq m < k_{q,r_{q,j}}$ such that 
$x \in T^{-r_{q,j}-m}q$.  Let $\ell_x = k_{q,r_{q,j}}-1$ which is greater than $\frac{2}{3} a_n-1$.  
Consider the ergodic average of $T$ for $x\in U_{q,j}^{\prime}$.  
Thus, 
\begin{align}
\sum_{i=0}^{\ell_x-1} f(T^{i} x) &= \sum_{i=0}^{m_{q,j}} f(T^{i}x) + \sum_{i=m_{q,j}+1}^{\ell_x-1} f(T^{i}x) \\ 
&> \frac{1}{2} \Big( \frac{3\ell_x}{4} \Big) - \frac{\ell_x}{4} > \frac{1}{8}\ell_x . 
\end{align}
Therefore, by Lemma \ref{lem-bad-fun}, this theorem follows.  $\Box$ 

\section{\bf Appendix A: Universal Moving Averages Via Coboundary Solutions}
\label{moving-avg-app}
Given a measurable function $f$, if there exist solutions $T$ and $g$ 
to the coboundary equation $f = g - g\circ T$ where $T$ is ergodic 
measure preserving and $g\in L^p(\mu)$ for $p>0$, then all 
moving averages $(v_n,L_n)$ converge a.e. where 
$L_n \geq n^{\frac{1}{p}}$.  A proof for the case $p=1$ 
is provided in the appendix of \cite{AR3}. 
\begin{thm}
\label{moving-avg-thm}
Suppose $T$ is an ergodic invertible measure preserving transformation 
on $(X, \B, \mu)$.  If $f$ is a coboundary with transfer function 
$g\in L^p(\mu)$ for $p>0$, then given $L_n \geq n^{\frac{1}{p}}$ and 
$v_n \in \mathbb{Z}$, there exists a set $E$ of measure zero such that 
for $x\notin E$: 
\[
\lim_{n\to \infty} \frac{1}{L_n} \sum_{i=1}^{L_n} f(T^{v_n + i} x) = 0 . 
\]
\end{thm}
{\bf Proof}: 
Suppose $f(x) = g(x) - g(Tx)$ for almost every $x\in X$.  Then 
\[
\sum_{i=1}^{L_n} f(T^{v_n + i} x) = g(T^{v_n+1}x) - g(T^{v_n + L_n + 1}x) . 
\]
Since $L_n \geq n^{\frac{1}{p}}$, it is sufficient to show each of the following:
\begin{itemize} 
\item for a.e. $x$, $\lim_{n\to \infty} {g(T^{v_n+1}x)}/{n^{\frac{1}{p}}} = 0$; 
\item for a.e. $x$, $\lim_{n\to \infty} {g(T^{v_n+L_n+1}x)}/{n^{\frac{1}{p}}} = 0$. 
\end{itemize}
Here, we show the first term converges to zero.  A similar argument will show the second term converges to zero.  

For $n, k\in \natural$, define 
\[
E_{n,k} = \{ x\in X: n-1 \leq | g(T^{v_k+1}x) |^p < n \} . 
\]
Since $g \in L^p(\mu)$, then 
\[
\sum_{n=1}^{\infty} (n-1) \mu (E_{n,1}) \leq \int_X | g(T^{v_1+1}x) |^p d\mu < \infty . 
\]
There exists non-decreasing $K_n \in \natural$ such that $\lim_{n\to \infty} {n}/{K_n} = 0$ and 
\[
\sum_{n=1}^{\infty} K_n \mu (E_{n,1}) < \infty . 
\]
Define 
\[
E = \bigcap_{m=1}^{\infty} \bigcup_{n=m}^{\infty} \bigcup_{k=1}^{K_n} E_{n,k} . 
\]
First, since $T$ is measure-preserving, we have 
\begin{align*}
\mu(E) \leq \sum_{n=m}^{\infty} \sum_{k=1}^{K_n} \mu(E_{n,k}) 
&= \sum_{n=m}^{\infty} \sum_{k=1}^{K_n} \mu(E_{n,1}) \\ 
&= \sum_{n=m}^{\infty} K_n \mu(E_{n,1}) \to 0\ \ \mbox{as}\ m\to \infty .  
\end{align*}
Thus, $\mu(E)=0$.  Second, we show for $x\notin E$, we have a.e. convergence.  
If $x\notin E$, there exists $m$ sufficiently large such that 
\[
x \notin \bigcup_{k=1}^{K_n} E_{n,k} \ \ \mbox{for}\ n\geq m . 
\]
Hence, if $n > |g(T^{v_k+1}x)|^p \geq n-1$, then $k > K_n$.  Therefore, 
\[
\frac{|g(T^{v_k+1}x)|}{k^{\frac{1}{p}}} < \Big( \frac{n}{K_n} \Big)^{\frac{1}{p}} 
\to 0\ \ \mbox{as}\ n\to \infty . \ \Box
\]

\begin{cor}
Suppose $(X,\B,\mu)$ is a standard probability space and $f \in L^p(\mu)$ 
for $p>1$.  
There exists an ergodic invertible measure preserving transformation 
$T$ on $(X,\B,\mu)$ such that all moving averages $(v_n,L_n)$ converge 
to zero pointwise where $L_n \geq n^{\frac{1}{p-1}}$.  
\end{cor}
{\bf Proof}: 
By Theorem 1.1 of \cite{AR3}, there exists a solution pair 
$T$ and $g \in L^{p-1}(\mu)$ such that $f = g - g\circ T$ a.e.  
Therefore, by Theorem \ref{moving-avg-thm}, all moving averages 
converge to zero for $f$ and $T$. $\Box$ 

\noindent {\bf Acknowledgments}:  We would like to thank Joanna Furno and Andy Parrish for their help in reviewing this work, and we thank Benjy Weiss for telling us how to know if a function $f \in L^1(\mu)$ can be the generating function of an IID sequence $(f\circ T^k)$ for some map $T$ on $X$.

\bigskip

{\small
\parbox[t]{5in}
{T. Adams\\
U.S. Government\\
E-mail: terry@ieee.org}}
\bigskip

{\small
\parbox[t]{5in}
{J. Rosenblatt \\
Department of Mathematics\\
University of Illinois\\
Urbana, IL 46201, USA\\
E-mail: rosnbltt@illinois.ed}}
\bigskip

\end{document}